\documentclass[a4paper]{article}
\usepackage{latexsym,bm}
\usepackage{amssymb,amsmath,amsthm}
\usepackage[english,german]{babel}
\usepackage[colorlinks=true]{hyperref}
\usepackage{color}

\allowdisplaybreaks[4]

\newtheorem{theorem}{Theorem}[section]
\newtheorem{lemma}[theorem]{Lemma}
\newtheorem{remark}[theorem]{Remark}
\newtheorem{proposition}[theorem]{Proposition}
\newtheorem{definition}[theorem]{Definition}
\newtheorem{corollary}[theorem]{Corollary}
\numberwithin{equation}{section}
\hypersetup{linkcolor=blue,urlcolor=red,citecolor=red}

\title{\bf On  Invariance of Observability for BSDEs and Its Applications to
  Stochastic Control Systems\thanks{This work was partially supported by National Natural Science Foundation of China
  under grant nos. 12131008 and U23B2033. }}

\author{Bao-Zhu Guo\thanks{Academy of Mathematics and Systems Science, Academia Sinica, China (bzguo@iss.ac.cn)}\and Huaiqiang Yu\thanks{School of Mathematics and KL-AAGDM, Tianjin University, Tianjin 300354, China (huaiqiangyu@tju.edu.cn, huaiqiangyu@yeah.net)}\and Meixuan Zhang\thanks{School of Mathematics and KL-AAGDM, Tianjin University, Tianjin 300354, China (zhangmx12023@163.com)}}

	\date{}
	
\begin{document}
\selectlanguage{english}
\maketitle

\begin{abstract}
     In this paper, we establish the  invariance of observability for the observed backward stochastic differential equations (BSDEs) with constant coefficients, relative to the filtered probability space. This signifies that the observability of these observed BSDEs with constant coefficients remains unaffected by the selection of the filtered probability space. As an illustrative application, we demonstrate that for stochastic control systems with constant coefficients, weak observability, approximate null controllability with cost, and stabilizability are equivalent across some or any filtered probability spaces.
\vskip 5pt
\noindent{\bf Keywords.} Stochastic control system, backward stochastic differential equation, observability, approximate null controllability with cost, stabilizability

\vskip 5pt

\noindent{\bf AMS subject classifications.} 93E03, 93B07, 93B05, 93D15

\end{abstract}

\section{Introduction}

     In this paper, we establish the invariance of observability for Backward Stochastic Differential Equations (BSDEs) under changes to the filtered probability space. This property fundamentally relies on the uniqueness of weak solutions (which is one kind of uniqueness in the sense of probability law) for Stochastic Differential Equations (SDEs) and forward-backward stochastic differential equations (FBSDEs).

The uniqueness of weak solutions is a pivotal and classical topic in the theory of stochastic differential equations. Initially, the uniqueness of weak solutions for SDEs was proven in \cite{Yamada-Watanabe-1971}, where the authors demonstrated that pathwise uniqueness of strong solutions implies uniqueness of weak solutions. For additional references in this field, one may refer, for example, to
\cite{Bass-Perkins-2012, Gomez-Lee-Mueller-Neuman-Salins-2017, Graczyk-Malecki-2013, Gkurtz-2007, Rocker-Schmuland-Zhang}, and the references  therein. With the progression of BSDE and FBSDE theories, the uniqueness of their weak solutions has been extensively studied over the past two decades. Notable references in this domain include \cite{Buckdahn-Engelbert-Rascanu-2005, Ma-Zhang-2011, Ma-Zhang-Zheng-2008, Yong-2023}, and the related citations. To our knowledge, no literature has yet utilized these uniqueness results to investigate the invariance of (weak or initial) observability of BSDEs concerning the filtered probability space. However, this invariance is crucial in characterizing the stabilizability of stochastic control system  through observability.

For nearly three decades, controllability and  stabilizability are pivotal topics  in control theory, particularly within the realm of stochastic control systems (refer to, for instance, \cite{ Lv-Yong-Zhang-2012, Lv-Zhang, Peng-1994, Rami-Zhou-2000, Sun-Yong-2017, Willems-Willems-1976, Zhang-Chen-2004}, and the associated references). Specifically, in \cite[Chapter 6]{Lv-Zhang}, the authors conducted a comparative analysis between deterministic and stochastic systems and compiled numerous existing findings on the controllability of stochastic control systems. These results indicate that the controllability of stochastic control systems remains an area of ongoing research. Indeed, due to the influence of random disturbances, stochastic control systems inherently possess infinite-dimensional characteristics. Consequently, determining whether such systems possess criteria analogous to the Kalman rank condition (pertinent to deterministic control systems) emerges as an intriguing and significant problem.
 Unfortunately,  such criteria are only applicable to certain specialized stochastic control systems (see \cite[Chapter 6, Theorems 6.8 and 6.10]{Lv-Zhang}). It is noteworthy that in \cite{Lv-Zhang}, the authors primarily examined null controllability and exact controllability. Hence, exploring more relaxed notions of controllability may offer enhanced algebraic characterizations. Recently, in \cite{Trelat-Wang-Xu}, the authors introduced the concept of approximate null controllability with cost to delineate the exponential stabilizability of deterministic infinite-dimensional control systems (also addressed in \cite{Liu-Wang-Xu-Yu, Ma-Wang-Yu}). This result motivates us to characterize the approximate null controllability with cost of stochastic control systems through their stabilizability.

     The primary contributions and novelties  of this paper can be summarized as follows: a) We demonstrate that,  for a BSDE  with constant coefficients, its (weak and initial) observability is independent of the chosen filtered probability space; b) We introduce the concept of approximate null controllability with cost and establish its equivalence to the corresponding stabilizability of the stochastic control system. Consequently, the stabilizability of the stochastic control system, which is inherently a property extending to infinite horizons, can be translated into a property confined to finite horizons: approximate null controllability with cost and  the conclusions on the the stabilizability of the stochastic control system con be used to characterize its approximate null controllability with cost; c) We transform the notion of approximate controllability with cost for stochastic control systems across any filtered probability spaces into the weak observability of the dual-observation system within a fixed filtered probability space. Furthermore, we present an algebraic condition to the approximate null controllability with cost associated with this transformed property.

    The structure of this paper is organized as follows: In Section \ref{Se2}, we introduce preliminary results that form the basis for our subsequent analyses. Section \ref{Se3} presents the problem formulation and the main results. Sections \ref{Se4} and \ref{Se5} are dedicated to proving the main results. Lastly, Section \ref{yu-sec-06} contains the Appendix, featuring any additional or supplementary materials relevant to our discussion.

{\bf Notations.}
    Let $\mathbb{R}^+:=(0,+\infty)$ and  $\mathbb{N}^+:=\mathbb{N}\setminus\{0\}$.
      Given
     a fixed measurable space $(\Omega,\mathcal{F})$, when a probability measure $\mathbb{P}$ is specified, we denote the probability space by $(\Omega,\mathcal{F},\mathbb{P})$ and  the expectation by  $E_{\mathbb{P}}$.  Suppose $w=(w^1,w^2,\cdots,w^d)^\top$ is a $d$-dimensional standard Brownian motion on $(\Omega,\mathcal{F},\mathbb{P})$. Let $\{\mathcal{F}^w_t\}_{t\geq 0}$
       be the natural filtration generated by $w(\cdot)$.  For simplicity,  we denote the filtered probability space
      $(\Omega,\mathcal{F},
    \{\mathcal{F}^w_t\}_{t\geq 0},\mathbb{P})$ by  $\mathfrak{T}(w,\mathbb{P})$.  Define
     $$
    L^2_{\mathcal{F}^w_T}(\Omega,\mathbb{P};\mathbb{R}^n):
    =\left\{f:\Omega\to \mathbb{R}^n: f
    \;\mbox{is}\;\mathcal{F}^w_T\mbox{-measurable}, E_{\mathbb{P}}
    |f|^2_{\mathbb{R}^n}<+\infty\right\}.
    $$
    For an interval $I\subset[0,+\infty)$, let
\begin{equation*}\label{yu-b-6-28-1}
    L^2_{\mathfrak{T}(w,\mathbb{P})}(I;\mathbb{R}^n):=\left\{\xi\in L^2(\Omega\times I,\mathbb{P}\times \lambda;\mathbb{R}^n):\xi
    \;\mbox{is}\;\{\mathcal{F}^w_t\}_{t\geq 0}\mbox{-progressively measurable}\right\},
\end{equation*}
\begin{equation*}\label{yu-b-6-28-2}
    C_{\mathfrak{T}(w,\mathbb{P})}(I;\mathbb{R}^n):=\left\{\xi:I\times\Omega
    \to\mathbb{R}^n:\xi
    \;\mbox{is}\;\{\mathcal{F}^w_t\}_{t\geq 0}\mbox{-adapted, contunous from}\;I\;\mbox{to}\; L^2_{\mathcal{F}^w_T}(\Omega,\mathbb{P};\mathbb{R}^n)\right\}.
\end{equation*}
     Here and throughout the remainder of this paper,  $\lambda$ denotes the Lebesgue measure on $\mathbb{R}$. It is well-known that $L^2_{\mathfrak{T}(w,\mathbb{P})}(I;\mathbb{R}^n)$
     equipped with the norm $\|\xi\|_{L^2_{\mathfrak{T}(w,\mathbb{P})}(I;\mathbb{R}^n)}:=\big(E_{\mathbb{P}}
    \int_I|\xi(t)|^2_{\mathbb{R}^n}dt\big)^{\frac{1}{2}}$
     forms a Hilbert space. Furthermore, for every  $T>0$, $C_{\mathfrak{T}(w,\mathbb{P})}([0,T];\mathbb{R}^n)$
    endowed with  norm $\|\xi\|_{C_{\mathfrak{T}(w,\mathbb{P})}([0,T];\mathbb{R}^n)}
    :=(\sup_{t\in[0,T]}E_{\mathbb{P}}|\xi(t)|^2_{\mathbb{R}^n})^{\frac{1}{2}}$,  constitutes a Banach space.  For a given matrix  $M$, $M^\top$
  denotes its transposition and $M_j$   represents its $j$-th column. We define
  $\mathbb{S}^n:=\{M\in\mathbb{R}^{n\times n}: M=M^\top,\;M\geq 0\}$ and   $\mathbb{S}^n_+:=\{M\in\mathbb{R}^{n\times n}: M=M^\top,\;M> 0\}$.
       For  $a,b\in\mathbb{R}$, we define $a\vee b=\max\{a,b\}$ and $a\wedge b=\min\{a,b\}$ . Throughout this paper, $c(\cdots)$  denotes a positive constant that depends on the elements within the brackets, which may vary across different lines.

      Let $\mathfrak{T}(w,\mathbb{P})$  be fixed arbitrarily. Throughout this paper,  we denote by $[A,B,C,D,\mathfrak{T}(w,\mathbb{P})]$ the following stochastic control system on $\mathfrak{T}(w,\mathbb{P})$:
\begin{equation}\label{yu-6-14-1}
    dx(t)=(Ax(t)+Bu(t))dt+\sum_{i=1}^d(C_ix(t)+D_iu(t))dw^i(t),\;\;t\in\mathbb{R}^+,
\end{equation}
    where $u\in L^2_{\mathfrak{T}(w,\mathbb{P})}(\mathbb{R}^+;\mathbb{R}^m)$, $A\in\mathbb{R}^{n\times n}$, $B\in\mathbb{R}^{n\times m}$, $C:=\{C_i\}_{i=1}^d$ and $D:=\{D_i\}_{i=1}^d$ with $C_i\in\mathbb{R}^{n\times n}$ and $D_i\in\mathbb{R}^{n\times m}$, respectively, for each $i\in\{1,2,\ldots,d\}$.    It is well-known that, given the initial value  $x_0$ (i.e. $x(0)=x_0$) and the control $u$, the equation \eqref{yu-6-14-1}    admits a unique solution in  $C_{\mathfrak{T}(w,\mathbb{P})}(\mathbb{R}^+;\mathbb{R}^n)$. We denote this unique solution by $x(\cdot;u,x_0)$ or $x^w(\cdot;u,x_0)$  to emphasize its dependence on the filtered probability space.
   \section{Preliminaries}\label{Se2}

In this section,  we present some fundamental results pertaining to BSDEs, including the explicit expression of a certain type of FBSDEs. Additionally, we discuss the  Linear-Quadratic (LQ, for simplicity) problems over infinite horizons.

\subsection{Backward shift property  for BSDEs}

    Let $T>0$ and fix an arbitrary filtered probability space $\mathfrak{T}(w,\mathbb{P})$.
    For any $k\in\mathbb{N}^+$, define  $W^d(0,kT):=C([0,kT];\mathbb{R}^d)$ (with the norm of $\sup_{t\in[0,kT]}|f(t)|_{\mathbb{R}^d}$ for $f\in W^d(0,kT)$) and $W^d_t(0,kT):=
    \{\xi( t\wedge\cdot):\xi(\cdot)\in W^d(0,kT)\}$ for all  $ t\in [0,kT]$.  For every  $t\in[0,kT]$, denote the set of all Borel subsets of
   $W^d_t(0,kT)$ by $\mathcal{B}(W^d_t(0,kT))$  and let $\mathcal{B}_t(W^d(0,kT)):=\sigma(\mathcal{B}(W^d_t(0,kT))$
    be the $\sigma$-field generated by $\mathcal{B}(W^d_t(0,kT))$.  Define
    $\mathcal{B}_{t+}(W^d(0,kT)):=\cap_{s>t} \mathcal{B}_s(W^d(0,kT))$
     and introduce the following set:
 \begin{equation}\label{yu-6-27-00}
  \mathcal{A}_{kT}^{d,n}:=\left\{\eta:[0,kT]\times W^d(0,kT)\to \mathbb{R}^n: \eta\;\mbox{is}\;
    \left\{\mathcal{B}_{t+}(W^d(0,kT))\right\}_{t\geq 0}\mbox{-progressively measurable}\right\}.
\end{equation}

\begin{lemma}\label{yu-lemma-6-27-1}
    Let $k\geq 2$. Suppose that $y_1\in L^2_{\mathcal{F}^w_{kT}}(\Omega,\mathbb{P};\mathbb{R}^n)$, $y\in C_{\mathfrak{T}(w,\mathbb{P})}([0,kT];\mathbb{R}^n)$, $f\in L^2_{\mathfrak{T}(w,\mathbb{P})}(0,kT;\mathbb{R}^n)$  and $\{g_i\}_{i=1}^d\subset L^2_{\mathfrak{T}(w,\mathbb{P})}(0,kT;\mathbb{R}^n)$ satisfy the equation
\begin{equation}\label{yu-6-25-2}
    y(t)=y_1+\int_t^{kT}f(s)ds+\sum_{i=1}^d\int_t^{kT}g_i(s)dw^i(s)
    \mbox{ for every }t\in[0,kT].
\end{equation}
     Define $\widehat{w}(\cdot):=w((k-1)T+\cdot)-w((k-1)T)$  in $\mathbb{R}^+$. Then, $\widehat{y}(\cdot):=y((k-1)T+\cdot)$,
     $\widehat{f}(\cdot):=f((k-1)T+\cdot)$ and $\widehat{g}_i(\cdot)=g_i((k-1)T+\cdot)$ $(i=1,2,\ldots,d)$  are adapted to $\{\mathcal{G}^{\widehat{w}}_t\}_{t\in[0,T]}$ and satisfy the equation
\begin{equation*}\label{yu-6-25-3}
   \widehat{y}(t)=y_1+\int_t^{T}\widehat{f}(s)ds+\sum_{i=1}^d
   \int_t^{T}\widehat{g}_i(s)d\widehat{w}^i(s)
    \;\;\mbox{for}\;\;t\in[0,T]
\end{equation*}
     on the filtered probability space $\mathfrak{T}(\widehat{w},\mathbb{P}
    (\cdot|\mathcal{F}^w_{(k-1)T})(\omega))$ for  $\mathbb{P}$-almost surely (a.s.)  $\omega\in \Omega$,
    where $\{\mathcal{G}^{\widehat{w}}_t\}_{t\in[0,T]}$ is the natural filtration generated by
    $\widehat{w}(\cdot)$ on $(\Omega,\mathcal{F},\mathbb{P}
    (\cdot|\mathcal{F}^w_{(k-1)T})(\omega))$.  Noted that
    $\widehat{w}(\cdot)$ is the standard Brownian motion in $(\Omega,\mathcal{F}, \mathbb{P}(\cdot|\mathcal{F}^w_{(k-1)T})(\omega))$ for $\mathbb{P}$-a.s. $\omega\in \Omega$.
\end{lemma}
\begin{proof}[\emph{\textbf{Proof.}}]
    By \cite[Chapter 1, Proposition 2.13]{Yong-Zhou-1999}, there is a $\mathbb{P}$-null set $\mathcal{N}\in \mathcal{F}$ such that, for every  $\omega_0\in \Omega\setminus \mathcal{N}$,
   \begin{equation}\label{yu-6-25-4}
        w(s,\omega)=w(s,\omega_0),\;\;
 \mathbb{P}(\cdot|\mathcal{F}^w_{(k-1)T})(\omega_0)
    \mbox{-a.s.}\;\;\omega\in\Omega\;\;\mbox{for each}\;\;s\in[0,(k-1)T].
\end{equation}
    Fix an arbitrary  $\omega_0\in \Omega\setminus \mathcal{N}$.  Since $y$,
    $f$ and
    $\{g_i\}_{i=1}^d$ are $\{\mathcal{F}^w_t\}_{t\in[0,kT]}$-adapted, according to  \cite[Chapter 1, Theorem 2.10]{Yong-Zhou-1999}, there exit  $\varphi$, $\xi$ and $\{\psi_i\}_{i=1}^d$ in $\mathcal{A}_{kT}^{d,n}$ (defined by \eqref{yu-6-27-00})  such that, for any  $t\in [0,kT]$,
\begin{equation}\label{yu-6-25-10}
\begin{cases}
    y(t,\omega)=\varphi(t,w(t\wedge \cdot,\omega))\\
    f(t,\omega)=\xi(t,w(t\wedge \cdot,\omega))\\
    g_i(t,\omega)=\psi_i(t,w(t\wedge \cdot,\omega))
    \;\;\mbox{for}\;\;i\in\{1,2,\ldots,d\},
\end{cases}
    \;\;\mathbb{P}\mbox{-a.s.}\;\;
    \omega\in\Omega.
\end{equation}
 From \eqref{yu-6-25-2}, we deduce that, for any  $t\in[0,T]$,
\begin{eqnarray}\label{yu-6-25-11}
    &\;&\varphi((k-1)T+t,w(((k-1)T+t)\wedge \cdot,\cdot))\nonumber\\
    &=&
    \varphi(kT, w((kT)\wedge \cdot,\cdot))+
    \int_{(k-1)T+t}^{kT}\xi(s,w(s\wedge \cdot,\cdot))ds+\sum_{i=1}^d
    \int_{(k-1)T+t}^{kT}\psi_i(s,w(s\wedge \cdot,\cdot))dw_i(s).
\end{eqnarray}
   (Here and in what follows, for $a\in\mathbb{R}^+$,  $w(a\wedge\cdot,\cdot)$
   and $\widehat{w}(a\wedge((\cdot-(k-1)T)\vee 0),\cdot)$ stand for two functions in $[0,kT]\times \Omega$.  In them, the first dot represents the time variable in $[0,kT]$ and the second dot represents the sample variable in $\Omega$.) It is straightforward to verify that for a fixed  $t\in\mathbb{R}^+$,
\begin{equation*}
    ((k-1)T+t)\wedge s-(k-1)T
    =
\begin{cases}
        s-(k-1)T&\mbox{if}\;\;s\leq (k-1)T+t,\\
    t&\mbox{if}\;\;s>(k-1)T+t.
\end{cases}
\end{equation*}
     Consequently,
\begin{equation}\label{yu-6-16-000-bb}
    ((k-1)T+t)\wedge s-(k-1)T=t\wedge(s-(k-1)T)\in [0,t]\;\;\mbox{for every}\;\;s\in [(k-1)T,kT].
\end{equation}
   It follows from the definition of $\widehat{w}(\cdot)$  that
\begin{eqnarray}\label{yu-6-25-12}
    &\;&\{\varphi((k-1)T+t,w(((k-1)T+t)\wedge \cdot,\cdot))\}_{t\in[0,T]}\nonumber\\
    &=&\{\varphi((k-1)T+t,\widehat{w}(t\wedge((\cdot-(k-1)T))\vee 0),\cdot)+w(((k-1)T)\wedge \cdot,\cdot))\}_{t\in[0,T]}.
\end{eqnarray}
     (Here, it should be noted that, by the definition of $\widehat{w}$, $\widehat{w}(t\wedge((\cdot-(k-1)T)\vee 0),\omega)$ is continuous in $[0,kT]$ for $\mathbb{P}$-a.s. $\omega\in \Omega$ and $(t\wedge((s-(k-1)T)))\vee 0=t\wedge((s-(k-1)T))\vee 0)$ for each $s\in[0,kT]$.)  Similarly, we can demonstrate that
\begin{eqnarray}\label{yu-6-26-0001}
    &\;&\{\xi((k-1)T+t,w(((k-1)T+t)\wedge \cdot,\cdot))\}_{t\in[0,T]}\nonumber\\
    &=&\{\xi((k-1)T+t, \widehat{w}(t\wedge((\cdot-(k-1)T)\vee 0),\cdot)+w(((k-1)T)\wedge \cdot,\cdot))\}_{t\in[0,T]}
\end{eqnarray}
    and
\begin{eqnarray}\label{yu-6-25-13}
    &\;&\{\psi_i((k-1)T+t,w(((k-1)T+t)\wedge \cdot,\cdot))\}_{t\in[0,T]}\nonumber\\
    &=&\{\psi_i((k-1)T+t,\widehat{w}(t\wedge((\cdot-(k-1)T)\vee 0),\cdot)+w(((k-1)T)\wedge \cdot,\cdot))\}_{t\in[0,T]}
\end{eqnarray}
 for $i\in\{1,2,\ldots,d\}$.
 Moreover, by \eqref{yu-6-25-13} and the definition of It\^{o}'s integral, it is straightforward to verify that, for each
 $i\in\{1,2,\ldots,d\}$ and $t\in[0,T]$,
\begin{eqnarray}\label{yu-6-26-1}
    &\;&\int_{(k-1)T+t}^{kT}\psi_i(s,w(s\wedge \cdot,\cdot))dw_i(s)\nonumber\\
    &=&\int_{t}^{T}\psi_i((k-1)T+s,\widehat{w}(s\wedge((\cdot-(k-1)T)\vee 0),\cdot)+w(((k-1)T)\wedge \cdot,\cdot))d\widehat{w}_i(s).
\end{eqnarray}
     Furthermore, by \eqref{yu-6-26-0001}, for any $t\in[0,T]$, we have
\begin{eqnarray}\label{yu-6-26-5}
    &\;&\int_{(k-1)T+t}^{kT}\xi(s,w(s\wedge \cdot,\cdot))ds\nonumber\\
    &=&\int_{t}^{T}\xi((k-1)T+s,\widehat{w}(s\wedge((\cdot-(k-1)T)\vee 0),\cdot)+w(((k-1)T)\wedge \cdot,\cdot))ds.
\end{eqnarray}
    For any $t\in[0,T]$, we define
\begin{equation}\label{yu-6-26-2}
\begin{cases}
    \widetilde{y}(t,\cdot):=\varphi((k-1)T+t, \widehat{w}(t\wedge((\cdot-(k-1)T)\vee 0),\cdot)+w(((k-1)T)\wedge \cdot,\cdot)),\\
    \widetilde{f}(t,\cdot):=\xi((k-1)T+t, \widehat{w}(t\wedge((\cdot-(k-1)T)\vee 0),\cdot)+w(((k-1)T)\wedge \cdot,\cdot)),\\
    \widetilde{g}_i(t,\cdot):=\psi_i((k-1)T+t, \widehat{w}(t\wedge((\cdot-(k-1)T)\vee 0),\cdot)+w(((k-1)T)\wedge \cdot,\cdot)),&i\in\{1,2,\ldots,d\}.
\end{cases}
\end{equation}
    Fix $t\in[0,T]$ arbitrarily. By \eqref{yu-6-25-4}, there is $\mathcal{M}:=\mathcal{M}(\omega_0)
     \subset\Omega$ with $\mathbb{P}(\mathcal{M}|\mathcal{F}_{(k-1)T}^w)(\omega_0)=0$ such that
     $w(s,\omega)=w(s,\omega_0)$ for any $\omega\in\Omega\setminus\mathcal{M}$ and each $s\in[0,(k-1)T]$. Thus, for each $t\in[0,T]$,
\begin{eqnarray}\label{yu-4-28-5}
    \widetilde{y}(t,\cdot)=\varphi((k-1)T+t, \widehat{w}(t\wedge((\cdot-(k-1)T)\vee 0),\cdot)+w(((k-1)T)\wedge \cdot,w_0))\;\;\mbox{on}\;\;
    \Omega\setminus\mathcal{M}.
\end{eqnarray}
    By the Borel measurability of $\varphi(t,\cdot)$, it is clear that
\begin{eqnarray}\label{yu-4-28-1}
        \varphi((k-1)T+t,\widehat{w}(t\wedge((\cdot-(k-1)T)\vee 0),\cdot)+w(((k-1)T)\wedge \cdot,w_0))\;\mbox{is}\; \mathcal{G}_t^{\widehat{w}}\mbox{-measurable}.
\end{eqnarray}
    (Here, we note that $w(((k-1)T)\wedge \cdot,w_0)$ is a known function (in $[0,kT]$). Moreover, since $\mathbb{P}(\mathcal{M}|\mathcal{F}_{(k-1)T}^w)(\omega_0)=0$, we can claim that
    $\mathcal{M}$ and $\Omega\setminus\mathcal{M}$ are in $\mathcal{G}_t^{\widehat{w}}$.
    Taking Borel set $I$ of $\mathbb{R}^n$ and fix $t\in[0,T]$ arbitrarily. Define
\begin{eqnarray*}
    E^I(\omega_0):=\{\omega\in \Omega:\varphi((k-1)T+t, \widehat{w}(t\wedge((\cdot-(k-1)T)\vee 0),\omega)+w(((k-1)T)\wedge \cdot,\omega_0))\in I\},
\end{eqnarray*}
\begin{eqnarray*}
     E_1^I(\omega_0):=\{\omega\in \Omega\setminus\mathcal{M}:\varphi((k-1)T+t, \widehat{w}(t\wedge((\cdot-(k-1)T)\vee 0),\omega)+w(((k-1)T)\wedge \cdot,\omega_0))\in I\}
\end{eqnarray*}
    and
\begin{eqnarray*}
    E^I_2(\omega_0):=\{\omega\in \mathcal{M}:\varphi((k-1)T+t, \widehat{w}(t\wedge((\cdot-(k-1)T)\vee 0),\omega)+w(((k-1)T)\wedge \cdot,,\omega_0))\in I\}.
\end{eqnarray*}
    It is obvious that $E_1^I(\omega_0)=E^I(\omega_0)\cap (E_2^I(\omega_0))^c$ and $(E^I_2(\omega_0))^c$ is in $\mathcal{G}_t^{\widehat{w}}$ (because of $E^I_2(\omega_0)\subset \mathcal{M}$ and $\mathbb{P}(\mathcal{M}|\mathcal{F}_{(k-1)T}^w)(\omega_0)=0$).
    These, along with \eqref{yu-4-28-1}, give that
\begin{equation}\label{yu-4-28-3}
    E_1^I(\omega_0)\;\;\mbox{is in}\;\;\mathcal{G}_t^{\widehat{w}}.
\end{equation}
    Moreover, by \eqref{yu-4-28-5}, we have
\begin{eqnarray}\label{yu-4-28-2}
\{\omega\in \Omega:\widetilde{y}(t,\omega)\in I\}
=\{\omega\in \Omega\setminus\mathcal{M}:\widetilde{y}(t,\omega)\in I\}\cup \{\omega\in\mathcal{M}:\widetilde{y}(t,\omega)\in I\}\nonumber\\
= E_1^I(\omega_0)\cup \{\omega\in\mathcal{M}:\widetilde{y}(t,\omega)\in I\}.
\end{eqnarray}
    Since $\{\omega\in\mathcal{M}:\widetilde{y}(t,\omega)\in I\}\subset\mathcal{M}$ and $\mathbb{P}(\mathcal{M}|\mathcal{F}_{(k-1)T}^w)(\omega_0)=0$, by \eqref{yu-4-28-3}
    and \eqref{yu-4-28-2}, we can claim that $\{\omega\in \Omega:\widetilde{y}(t,\omega)\in I\}$ is in $\mathcal{G}_t^{\widehat{w}}$. This, along with the arbitrariness of $I$, implies that $\widetilde{y}(t)$ is $\mathcal{G}_t^{\widehat{w}}$-measurable. By the same way, we can show that $\widetilde{f}(t)$ and $\widetilde{g}_i(t)$ ($i=1,2,\ldots,d$) are $\mathcal{G}_t^{\widehat{w}}$-measurable. Thus, by the arbitrariness of $t\in[0,T]$, we can conclude that $\widetilde{y}$, $\widetilde{f}$ and $\widetilde{g}_i$ ($i=1,2,\ldots,d$)  are adapted to $\{\mathcal{G}^{\widehat{w}}_t\}_{t\in[0,T]}$ on $(\Omega,\mathcal{F}, \mathbb{P}(\cdot|\mathcal{F}^w_{(k-1)T})(\omega_0))$.
     Consequently, by \eqref{yu-6-25-11}, \eqref{yu-6-26-1}, \eqref{yu-6-26-5}, and \eqref{yu-6-26-2}, we derive
$\widetilde{y}(t)=\widetilde{y}(T)+
    \int_t^{T}\widetilde{f}(s)ds
    +\sum_{i=1}^d\int_t^{T}\widetilde{g}_i(s)d\widehat{w}_i(s)$
     on $(\Omega,\mathcal{F}, \mathbb{P}(\cdot|\mathcal{F}^w_{(k-1)T})(\omega_0))$.
    However, by \eqref{yu-6-25-10}, it is clear that $\widetilde{y}(\cdot)=\widehat{y}(\cdot)$, $\widetilde{f}(\cdot)=\widehat{f}(\cdot)$ and $\widetilde{g}_i(\cdot)=\widehat{g}_i(\cdot)$ for $i\in\{1,2,\ldots,d\}$. This, combined with \eqref{yu-6-25-12}, \eqref{yu-6-26-0001}, \eqref{yu-6-25-13}, and the fact that $\mathbb{P}(\mathcal{N})=0$, concludes the proof.
\end{proof}

    Let $k\in\mathbb{N}^+$ be fixed arbitrarily. We consider the following  BSDE on $\mathfrak{T}(w,\mathbb{P})$:
\begin{equation}\label{yu-6-22-1}
\begin{cases}
    dy(t)=(-A^{\top}y(t)-\sum_{i=1}^dC_i^\top Y_i(t))dt+\sum_{i=1}^dY_i(t)dw^i(t),
    &t\in(0,kT),\\
    y(kT)=y_1\in L^2_{\mathcal{F}^w_{kT}}(\Omega,\mathbb{P};\mathbb{R}^n).
\end{cases}
\end{equation}
    The unique solution to this equation is denoted by  $(y^k(\cdot;y_1),(Y^k_i(\cdot;y_1))_{i=1}^d)$ which belongs to $C_{\mathfrak{T}(w,\mathbb{P})}([0,kT];\mathbb{R}^n)\times (L^2_{\mathfrak{T}(w,\mathbb{P})}(0,kT;\mathbb{R}^n))^d$.
    (For further details, see, for instance,  \cite[Chapter 4, Theorem 4.2]{Lv-Zhang}.)
\begin{proposition}\label{yu-proposition-6-27-1}
    Let $k\geq 2$ and define  $\widehat{w}(\cdot):=w((k-1)T+\cdot)-w((k-1)T)$ in $[0,T]$. Suppose that $y_1\in L^2_{\mathcal{F}^w_{kT}}(\Omega,\mathbb{P};\mathbb{R}^n)$. Then,
    for $\mathbb{P}$-a.s. $\omega\in \Omega$, on $\mathfrak{T}(\widehat{w},\mathbb{P}
    (\cdot|\mathcal{F}^w_{(k-1)T})(\omega))$, the following statements  are true:
\begin{enumerate}
  \item [$(i)$] $y_1\in L^2_{\mathcal{G}_T^{\widehat{w}}}(\Omega,
  \mathbb{P}
    (\cdot|\mathcal{F}^w_{(k-1)T})(\omega);\mathbb{R}^n)$.
  \item [$(ii)$] $z(\cdot):=y^k((k-1)T+\cdot;y_1)$ and $\{Z_i(\cdot)\} _{i=1}^d:=\{Y_i^k((k-1)T+\cdot;y_1)\}_{i=1}^d$ are
   $\{\mathcal{G}^{\widehat{w}}_t\}_{t\in[0,T]}$-adopted on $\mathfrak{T}(\widehat{w},\mathbb{P}
    (\cdot|\mathcal{F}^w_{(k-1)T})(\omega))$.
   \item[$(iii)$] ($z(\cdot):=y^k((k-1)T+\cdot;y_1), (Z_i(\cdot)) _{i=1}^d:=(Y_i^k((k-1)T+\cdot;y_1))_{i=1}^d)$ is the solution of the following BSDE (on $\mathfrak{T}(\widehat{w},\mathbb{P}
    (\cdot|\mathcal{F}^w_{(k-1)T})(\omega))$):
  \begin{equation}\label{yu-6-22-bb-1}
\begin{cases}
    dz(t)=(-A^{\top}z(t)-\sum_{i=1}^dC_i^\top Z_i(t))dt+\sum_{i=1}^dZ_i(t)d\widehat{w}^i(t),
    &t\in(0,T),\\
    z(T)=y_1.
\end{cases}
\end{equation}
\end{enumerate}
    Here  $\{\mathcal{G}^{\widehat{w}}_t\}_{t\in[0,T]}$ is the natural filtration generated by
    $\widehat{w}(\cdot)$ on $(\Omega,\mathcal{F},\mathbb{P}
    (\cdot|\mathcal{F}^w_{(k-1)T})(\omega))$.
\end{proposition}
\begin{proof}
    Let $y(\cdot):=y^k(\cdot;y_1)$,  $g_i(\cdot):=Y^k_i(\cdot;y_1)$ for
    $i\in\{1,2,\ldots,d\}$,  and define $f:=-A^\top y-\sum_{i=1}^dC_i^\top g_i$.
     One can readily verify that the tuple $(y,f,\{g_i\}_{i=1}^d)$  satisfies the assumptions outlined  in Lemma \ref{yu-lemma-6-27-1}.  Consequently, by invoking Lemma \ref{yu-lemma-6-27-1},
     we can affirm that the claim $(ii)$ and $(iii)$ hold on $\mathfrak{T}(\widehat{w},\mathbb{P}
    (\cdot|\mathcal{F}^w_{(k-1)T})(\omega))$ for $\mathbb{P}$-a.s. $\omega\in \Omega$. By $(ii)$ and the fact $z(T)=y_1$, the claim $(i)$ is obvious.
     Thus, the statements in Proposition \ref{yu-proposition-6-27-1} are true.
\end{proof}
\begin{remark}\label{yu-remark-4-11-1}
    Under the assumptions in Proposition \ref{yu-proposition-6-27-1}, by $(i)$ in Proposition \ref{yu-proposition-6-27-1} and \cite[Chapter 4, Theorem 4.2]{Lv-Zhang}, for $\mathbb{P}$-a.s. $\omega\in \Omega$, in $C_{\mathfrak{T}(\widehat{w},\mathbb{P}
    (\cdot|\mathcal{F}^w_{(k-1)T})(\omega))}([0,T];\mathbb{R}^n)\times (L^2_{\mathfrak{T}(\widehat{w},\mathbb{P}
    (\cdot|\mathcal{F}^w_{(k-1)T})(\omega))}(0,T;\mathbb{R}^n))^d$, the equation \eqref{yu-6-22-bb-1} has a unique solution.
\end{remark}
\subsection{A characterization on  a special FBSDE}\label{yu-sect-8-22}

  In this subsection and throughout the remainder of this paper, we define
  $C^{k_1,k_2}([0,T]\times \mathbb{R}^{p};\mathbb{R}^q)$ ($k_1,k_2, p,q\in\mathbb{N}$)
    as the set of functions from   $[0,T]\times \mathbb{R}^{p}$ to $\mathbb{R}^q$ that
     are of class $C^{k_1,k_2}$,
 with all partial derivatives of order less than or equal to  $k_1$
   in the time variable and less than or equal to  $k_2$
    in the space variables. Similarly, we define
    $C^k(\mathbb{R}^p;\mathbb{R}^q)$.  Furthermore, for a given
    $g=(g_1,g_2,\cdots,g_q)^\top\in C^1(\mathbb{R}^p;\mathbb{R}^q)$ (where  $(p,q\in\mathbb{N}^+)$,
      we denote by $\nabla g:=(\nabla g_1,\nabla g_2,\cdots,\nabla g_q)^\top$
       the gradient operator of $g$.

    Let $T\in\mathbb{R}^+$, $l\in\mathbb{N}^+$, $N\in\mathbb{N}^+$ and the filtered probability space $\mathfrak{T}(w,\mathbb{P})$ be fixed. For each $k\in\{1,2,\ldots,l\}$, we take $\alpha_k:=(\alpha_{k,1},\alpha_{k,2},\cdots,\alpha_{k,n})^\top\in\mathbb{R}^n$, $h_k:=(h_{k,1},h_{k,2},\cdots,h_{k,n})\in C([0,T];\mathbb{R}^{d\times n}$) and  $H_k:=\frac{1}{2}(|h_{k,1}|^2_{\mathbb{R}^d}, |h_{k,2}|^2_{\mathbb{R}^d},
    \cdots,|h_{k,n}|^2_{\mathbb{R}^d})$.
   We  let the function  $V_k:\mathbb{R}^n\to \mathbb{R}^n$ be defined by
    $V_k(z):=(\alpha_{k,1}\exp(z_{1}), \alpha_{k,2}\exp(z_{2}),\cdots, \alpha_{k,n}\exp(z_{n}))^\top$ for any
    $z=(z_{1},z_{2},\cdots,z_{n})^\top\in\mathbb{R}^n$.
     Let $F_l:\mathbb{R}^{nl}\to \mathbb{R}^n$ be defined by
\begin{equation}\label{yu-7-23-bbbb-1}
    F_l(x):=\sum_{k=1}^lV_k(x_k)\;\;\mbox{with}\;\;x:=(x_1^\top,x_2^\top,
    \cdots,x_l^\top)^\top(\in\mathbb{R}^{nl})
\end{equation}
    and
\begin{equation}\label{yu-7-11-3}
    F_{l,N}(x):=\psi_N(x)F_l(x)\;\;\mbox{for}\;\;x\in\mathbb{R}^{nl},
\end{equation}
    where $\psi_N\in C_c^\infty(\mathbb{R}^{nl};\mathbb{R})$ satisfies
    $\psi_N(x)\in[0,1]$ for any $x\in\mathbb{R}^{nl}$ and
$\psi_N(x)=1$ if $|x|_{\mathbb{R}^{nl}}\leq N$ and
    $\psi_N(x)=0$  if $|x|_{\mathbb{R}^{nl}}\geq 2N$.

    Define
\begin{equation}\label{yu-7-23-b0-001}
    H:=(H_1,H_2,\cdots,H_l)
    \;\;\mbox{and}\;\;h:=(h_1,h_2,\cdots,h_l).
\end{equation}
    It is evident that $H^\top \in C([0,T];\mathbb{R}^{nl})$ and
    $h\in C([0,T];\mathbb{R}^{d\times(nl)})$. For any  $s\in[0,T]$, let  $\mathcal{F}^{w,s}_t$ be the filtration generated by $w(t)-w(s) (t\geq s)$ and define $\mathfrak{T}^s(w,\mathbb{P}):=(\Omega,\mathcal{F},
    \{\mathcal{F}^{w,s}_t\}_{t\geq s},\mathbb{P})$.

    Let
    $(x^{s,\zeta}, y^{s,\zeta},Y^{s,\zeta}=(Y_i^{s,\zeta})_{i=1}^d)$  be the  solution of the following  FBSDE:
\begin{equation}\label{yu-7-9-3}
\begin{cases}
    dx(t)=-H^\top(t)dt+h^\top(t)dw(t),&t\in(s,T),\\
    dy(t)=(-A^{\top}y(t)-\sum_{i=1}^dC_i^\top Y_i(t))dt+\sum_{i=1}^dY_i(t)dw^i(t),
    &t\in(s,T),\\
    x(s)=\zeta\in\mathbb{R}^{nl},\;\;y(T)=F_{l,N}(x(T)).
\end{cases}
\end{equation}
   It is well known that
$(x^{s,\zeta}, y^{s,\zeta},Y^{s,\zeta})\in
     C_{\mathfrak{T}^s(w,\mathbb{P})}([s,T];\mathbb{R}^{nl})
     \times C_{\mathfrak{T}^s(w,\mathbb{P})}([s,T];\mathbb{R}^n)
     \times
     (L^2_{\mathfrak{T}^s(w,\mathbb{P})}(s,T;\mathbb{R}^n))^d$.
     Here, we note that $F_{l,N}\in C_c^\infty(\mathbb{R}^{nl};\mathbb{R}^n)$ and consequently, $F_{l,N}(x^{s,\zeta}(T))\in L^2_{\mathcal{F}^{w,s}_T}(\Omega,\mathbb{P};\mathbb{R}^n)$.
\begin{lemma}\label{yu-proposition-7-11-1}
     The $n$-dimensional parabolic equation:
\begin{equation}\label{yu-7-22-4}
\begin{cases}
    \partial_tU(t,\zeta)+[\mathcal{L}U](t,\zeta)+A^\top U(t,\zeta)+\sum_{i=1}^dC_i^\top
    \nabla U(t,\zeta)h^\top(t)f_i=0,&(t,\zeta)\in[0,T)\times \mathbb{R}^{nl},\\
    U(T,\zeta)=F_{l,N}(\zeta),&\zeta\in\mathbb{R}^{nl},
\end{cases}
\end{equation}
    admits a  unique bounded solution in $C^{1,2}([0,T]\times \mathbb{R}^{nl};\mathbb{R}^n)$,  where $(f_i)_{i=1}^d$ is the classical orthonormal basis of $\mathbb{R}^d$, and $\mathcal{L}\varphi=(L\varphi_1,L\varphi_2,\cdots,L\varphi_n)^\top$
    for $\varphi=(\varphi_1,\varphi_2,\cdots,\varphi_n)^\top\in C^2(\mathbb{R}^{nl};\mathbb{R}^n)$ with
$$L:=\frac{1}{2}
    \sum_{i,j}^{nl}(h^\top(t)h(t))_{i,j}\partial_{\zeta_i\zeta_j}^2+H^\top(t)
    \cdot\nabla.$$
    Furthermore, if $U$ is the solution of \eqref{yu-7-22-4} in $C^{1,2}([0,T]\times \mathbb{R}^{nl};\mathbb{R}^n)$, then
$y^{s,\zeta}(\cdot)=U(\cdot,x^{s,\zeta}(\cdot))$ and
    $Y^{s,\zeta}(\cdot)=\nabla U(\cdot,x^{s,\zeta}(\cdot))h^\top(\cdot)$ in $[s,T]$ for any $(s,\zeta)\in[0,T]\times \mathbb{R}^{nl}$.
\end{lemma}

      The proof of Lemma \ref{yu-proposition-7-11-1} follows a similar approach to the proofs of the main results presented in \cite{Pardoux-Peng-1990, Pardoux-Peng-1992}. For the sake of completeness in this paper, we provide a detailed proof in Appendix Section \ref{yu-sec-06}.

\subsection{LQ problem on infinite horizon}\label{yu-sec-2-2}
    In this subsection, we revisit some established results pertaining to the Linear-Quadratic (LQ) problem of stochastic control systems over an infinite horizon. We consider a fixed filtered probability space $\mathfrak{T}(w,\mathbb{P})$. For any control input
     $u\in  L^2_{\mathfrak{T}(w,\mathbb{P})}(\mathbb{R}^+;\mathbb{R}^m)$ and initial state $x_0\in\mathbb{R}^n$, we define the cost functional
     $\mathcal{J}(u, x_0)$  and introduce the set of admissible controls
     $\mathcal{U}_{ad}(x_0)$  as follows:
\begin{equation*}\label{yu-6-14-2}
    \mathcal{J}(u, x_0):=\frac{1}{2}E_{\mathbb{P}}
    \int_{\mathbb{R}^+}\left(|x(t;u,x_0)|_{\mathbb{R}^n}^2+
    |u(t)|_{\mathbb{R}^m}^2\right)dt
\end{equation*}
    and
$\mathcal{U}_{ad}(x_0):=\{v\in L^2_{\mathfrak{T}(w,\mathbb{P})}(\mathbb{R}^+;\mathbb{R}^m): x(\cdot;v,x_0)
    \in L^2_{\mathfrak{T}(w,\mathbb{P})}(\mathbb{R}^+;\mathbb{R}^n)\}$.
    Here, we recall that
    $x(\cdot;u,x_0)$ is the solution of the control system $[A,B,C,D,\mathfrak{T}(w,\mathbb{P})]$ with initial value $x_0$ and control $u$. It is possible that  $\mathcal{U}_{ad}(x_0)\neq L^2_{\mathfrak{T}(w,\mathbb{P})}(\mathbb{R}^+;\mathbb{R}^m)$ (even $\mathcal{U}_{ad}(x_0)=\emptyset$) for some $x_0$. If  $\mathcal{U}_{ad}(x_0)\neq \emptyset$ (or, equivalently, $\inf_{u\in L^2_{\mathfrak{T}(w,\mathbb{P})}(\mathbb{R}^+;\mathbb{R}^m)}\mathcal{J}(u, x_0)<+\infty$), we consider the following optimal control problem:


    \textbf{(LQ)$_{\mathfrak{T}(w,\mathbb{P})}$:} \emph{Find $u^*_{x_0}\in \mathcal{U}_{ad}(x_0)$ such that
\begin{equation}\label{yu-6-14-4}
    \mathcal{J}(u^*_{x_0}, x_0)=\mathcal{V}(x_0):=\inf_{u\in
    L^2_{\mathfrak{T}(w,\mathbb{P})}(\mathbb{R}^+;\mathbb{R}^m)}\mathcal{J}(u, x_0).
\end{equation}}
\begin{remark}
   The function $\mathcal{V}(\cdot)$ defined by \eqref{yu-6-14-4} is called to be the value function of  \textbf{(LQ)$_{\mathfrak{T}(w,\mathbb{P})}$}. For a given  $x_0\in\mathbb{R}^n$, if $\mathcal{U}_{ad}(x_0)\neq \emptyset$, we can demonstrate through standard arguments that there exists a unique $u^*_{x_0}\in \mathcal{U}_{ad}(x_0)$ that satisfies \eqref{yu-6-14-4} and then $\mathcal{V}(x_0)<+\infty$. If  $\mathcal{U}_{ad}(x_0)\neq \emptyset$ for any $x_0\in\mathbb{R}^n$,
     we refer to the problem as \textbf{(LQ)$_{\mathfrak{T}(w,\mathbb{P})}$} being solvable. In such cases, for a fixed
     $x_0\in\mathbb{R}^n$,  we denote the optimal trajectory as  $x^*(\cdot,x_0):=x(\cdot;x_0,u^*_{x_0})$.
\end{remark}

\begin{lemma}\label{yu-lemma-6-14-1}
    (\cite[Theorem 5.1]{Huang-Li-Yong-2015} and \cite[Theorem 3.2]{Sun-Yong-2017}) The following statements are equivalent:
\begin{enumerate}
  \item [$(i)$] The problem \textbf{(LQ)$_{\mathfrak{T}(w,\mathbb{P})}$} is solvable;

  \item [$(ii)$]  The Riccati equation:
   \begin{equation}\label{yu-6-14-5}
    PA+A^\top P+\sum_{i=1}^dC_i^\top PC_i+\Big(PB+\sum_{i=1}^dC_i^\top PD_i\Big)\Big(I+\sum_{i=1}^dD_i^\top PD_i\Big)^{-1}\Big(B^\top P+\sum_{i=1}^dD_i^\top PC_i\Big)=0
\end{equation}
    has a unique solution in $\mathbb{S}^n_+$.
\end{enumerate}

    Furthermore, if either of the above statements holds, then
     the value function of  \textbf{(LQ)$_{\mathfrak{T}(w,\mathbb{P})}$} is given by $\mathcal{V}(x_0)=\langle Px_0,x_0\rangle_{\mathbb{R}^n}$ for every $x_0\in\mathbb{R}^n$, where $P$ is the solution to the Riccati equation \eqref{yu-6-14-5}. Additionally, the optimal control
     $u^*_{x_0}$   and the optimal trajectory  $x^*(\cdot;x_0)$ are related as
    $u^*_{x_0}(t)=- (I+\sum_{i=1}^dD_i^\top PD_i)^{-1}(B^\top P+\sum_{i=1}^dD_i^\top PC_i)x^*(t,x_0)$ for a.e. $t\in\mathbb{R}^+$.
\end{lemma}
\begin{remark}\label{yu-remark-7-26-1}
    It is evident that the solvability of the Riccati equation \eqref{yu-6-14-5} is independent of the choice of the filtered probability space. Therefore, by combining Lemma \ref{yu-lemma-6-14-1} with \cite[Proposition 3.6]{Huang-Li-Yong-2015} (or \cite[Lemma 2.3]{Sun-Yong-Zhang-2021}), we can conclude that the stabilizability (as defined in Definition \ref{yu-defi-10-18-1} below) of the system \eqref{yu-6-14-1} is also independent of the choice of the filtered probability space.
\end{remark}

\section{Control problem and main results}\label{Se3}


\subsection{Observability}

    As is well-known, many control problems (such as controllability) for  the system
     $[A,B,C,D,\mathfrak{T}(w,\mathbb{P})]$ (i.e., \eqref{yu-6-14-1})  are closely related to the observability of the following dual-observation system (or observed BSDE) on the interval  $[0,T]$ (with $T>0$):
\begin{equation}\label{yu-7-26-10}
\begin{cases}
    (\mbox{Dual system})\;\;
\begin{cases}
     dy(t)=(-A^{\top}y(t)-\sum_{i=1}^dC_i^\top Y_i(t))dt+\sum_{i=1}^dY_i(t)dw^i(t),&t\in[0,T],\\
     y(T)=y_1\in L^2_{\mathcal{F}_T^{w}}(\Omega,\mathbb{P};\mathbb{R}^n);
\end{cases} \\
    (\mbox{Obseveration system})\;\;
z(t)=B^\top y(t)+\sum_{i=1}^d D_i^\top Y_i(t),\;\;\;\;t\in[0,T].
\end{cases}
\end{equation}
     In some cases, for simplicity, we denote the system in \eqref{yu-7-26-10} by
    $[A^\top,B^\top,C^\top,D^\top, \mathfrak{T}(w,\mathbb{P})]_T$ if $T$ is
     fixed, otherwise $[A^\top,B^\top,C^\top,D^\top, \mathfrak{T}(w,\mathbb{P})]$.
       Since the dual system in \eqref{yu-7-26-10} possesses a solution in
       $C_{\mathfrak{T}(w,\mathbb{P})}([0,T];\mathbb{R}^n)
    \times (L^2_{\mathfrak{T}(w,\mathbb{P})}(0,T;\mathbb{R}^n))^d$
     (see, for instance, \cite[Chapter 4, Theorem 4.2]{Lv-Zhang}), the function
      $z:[0,T]\times\Omega\to\mathbb{R}^m$ (also referred to as the output function) belongs to
      $L^2_{\mathfrak{T}(w,\mathbb{P})}(0,T;\mathbb{R}^m)$. Given
    $y_1\in L^2_{\mathcal{F}_T^{w}}(\Omega,\mathbb{P};\mathbb{R}^n)$,  we denote the solution of the dual-observation system \eqref{yu-7-26-10} by $(y(\cdot;y_1),Y(\cdot;y_1):=(Y_i(\cdot;y_1))_{i=1}^d,z(\cdot;y_1))$ or $(y^w(\cdot;y_1),Y^w(\cdot;y_1):=(Y^w_i(\cdot;y_1))_{i=1}^d,z^w(\cdot;y_1))$ to emphasize its dependence on the filtered probability space.

     We now introduce the concept of observability for the dual-observation system $[A^\top,B^\top,C^\top,D^\top, \mathfrak{T}(w,\mathbb{P})]_T$.

\begin{definition}\label{yu-definition-7-29-1}
    Given a fixed  $\mathfrak{T}(w,\mathbb{P})$.
    \begin{enumerate}
      \item [$(i)$]
           For a fixed $T>0$ and $\delta\in [0,1)$, the system
           $[A^\top,B^\top,C^\top,D^\top, \mathfrak{T}(w,\mathbb{P})]_T$
             is said to be  $\delta$-observable if there exists a constant  $c(\delta,T)>0$  such that
          \begin{eqnarray}\label{yu-6-22-4-001}
    E_{\mathbb{P}}|y(0;y_1)|^2_{\mathbb{R}^n}
    \leq c(\delta,T)E_{\mathbb{P}}\int_0^T|z(t)|^2_{\mathbb{R}^m}dt
    +\delta E_{\mathbb{P}}|y_1|_{\mathbb{R}^n}^2\;\;
    \mbox{for any}\;\;y_1\in L^2_{\mathcal{F}^w_T}(\Omega,\mathbb{P};\mathbb{R}^n).
\end{eqnarray}
         Specifically, if $\delta=0$,  the system $[A^\top,B^\top,C^\top,D^\top, \mathfrak{T}(w,\mathbb{P})]_T$
      is called  initially observable.

      \item [$(ii)$]
           If there exists a $T>0$ and $\delta\in(0,1)$ such that the system
           $[A^\top,B^\top,C^\top,D^\top, \mathfrak{T}(w,\mathbb{P})]_T$ is $\delta$-observable, then  $[A^\top,B^\top,C^\top,D^\top, \mathfrak{T}(w,\mathbb{P})]$
            is said to be  weakly observable. If there exists a $T>0$ such that
             $[A^\top,B^\top,C^\top,D^\top, \mathfrak{T}(w,\mathbb{P})]_T$
               is initially observable, then $[A^\top,B^\top,C^\top,D^\top, \mathfrak{T}(w,\mathbb{P})]$
                is called  initially observable.
    \end{enumerate}
\end{definition}
\begin{remark}
 Three points should be noted regarding Definition \ref{yu-definition-7-29-1}: (a)  The inequality \eqref{yu-6-22-4-001} is called the weak observability inequality of the system  $[A^\top,B^\top,C^\top,D^\top, \mathfrak{T}(w,\mathbb{P})]_T$,  $\delta$ and $c(\delta,T)$ in it are referred to as the observation error and the observability constant, respectively; (b)  By the Blumenthal 0-1 law, $y(0;y_1)$ in \eqref{yu-6-22-4-001} is deterministic $\mathbb{P}$-a.s.; (c) It is evident  that if $[A^\top,B^\top,C^\top,D^\top, \mathfrak{T}(w,\mathbb{P})]$
                is initially observable then it is also
                 weakly observable.
\end{remark}

    From Definition \ref{yu-definition-7-29-1}, the observability of the system
    $[A^\top,B^\top,C^\top,D^\top, \mathfrak{T}(w,\mathbb{P})]_T$     is significantly influenced by four key factors: (a) the filtered probability space $\mathfrak{T}(w,\mathbb{P})$; (b) the observation interval $[0,T]$; (c) the observation error $\delta$; and (d) the algebraic structure of
    $[A^\top,B^\top,C^\top,D^\top]$.  It is well-established that factor (d) can impact the observability of $[A^\top,B^\top,C^\top,D^\top, \mathfrak{T}(w,\mathbb{P})]_T$    (see, for example, \cite[Chapter 6, Theorem 6.8]{Lv-Zhang} and \cite{Peng-1994}). Furthermore, in \cite{ Liu-Wang-Xu-Yu, Ma-Wang-Yu, Trelat-Wang-Xu}, the influence of factors (b) and (c) on deterministic observation systems is discussed. By applying dual arguments, we can deduce that factors (b) and (c) have a similar impact on the observability of $[A^\top,B^\top,C^\top,D^\top, \mathfrak{T}(w,\mathbb{P})]_T$.
      Therefore, the primary objective of this paper is to determine \emph{whether factor (a) influences the observability of $[A^\top,B^\top,C^\top,D^\top, \mathfrak{T}(w,\mathbb{P})]_T$.}

    The first main result of this paper is stated as follows.
\begin{theorem}\label{yu-theorem-7-24-1}
    Let $T>0$ and $\delta\in[0,1)$. The following statement are equivalent:
\begin{enumerate}
  \item [$(i)$]  There exists a filtered probability space $\mathfrak{T}(w,\mathbb{P})$ such that the system $[A^\top,B^\top,C^\top,D^\top,\mathfrak{T}(w,\mathbb{P})]_T$ is  $\delta$-observable;
  \item [$(ii)$] For any filtered probability space $\mathfrak{T}(w,\mathbb{P})$, the system $[A^\top,B^\top,C^\top,D^\top,\mathfrak{T}(w,\mathbb{P})]_T$ is  $\delta$-observable.
\end{enumerate}
         Furthermore, if there exists a filtered probability space  $\mathfrak{T}(w,\mathbb{P})$ such that the system $[A^\top,B^\top,C^\top,D^\top,\mathfrak{T}(w,\mathbb{P})]_T$ is  $\delta$-observable, then the quantity
  \begin{eqnarray}\label{yu-7-24-10}
    \mathcal{O}_{op}(\delta,T,\mathfrak{T}(w,\mathbb{P}))
    :=\inf\Big\{c\geq 0:E_{\mathbb{P}}|y^w(0,y_1)|_{\mathbb{R}^n}^2
    \leq cE_{\mathbb{P}}\int_0^T|z^w(t;y_1)|^2_{\mathbb{R}^m}dt+\delta E_{\mathbb{P}}|y_1|^2_{\mathbb{R}^n},\nonumber\\
    \;\forall y_1\in L^2_{\mathcal{F}_T^w}(\Omega,\mathbb{P};\mathbb{R}^n)\Big\}
\end{eqnarray}
    is independent of $\mathfrak{T}(w,\mathbb{P})$.
\end{theorem}
\begin{remark}\label{yu-remark-7-26-1}
\begin{enumerate}
  \item [$(i)$]      Theorem \ref{yu-theorem-7-24-1} informs us that the $\delta$-observability of \eqref{yu-7-26-10} solely depends on the algebraic structure of
      $[A^\top,B^\top,C^\top,D^\top]$ and the pair $(T,\delta)$.
            Consequently, when considering the (weak/initial) observability of the system \eqref{yu-7-26-10}, we may abbreviate $[A^\top,B^\top,C^\top,D^\top,\mathfrak{T}(w,\mathbb{P})]_T$ as $[A^\top,B^\top,C^\top,D^\top]_T$.
              This characteristic is referred to as the invariance of observability of \eqref{yu-7-26-10} with respect to the filtered probability space.

  \item [$(ii)$]
  The constant $\mathcal{O}_{op}(\delta,T,\mathfrak{T}(w,\mathbb{P}))$
    defined by \eqref{yu-7-24-10} is termed the optimal  $\delta$-observability constant. It can be readily demonstrated that this constant exists if the system
    $[A^\top,B^\top,C^\top,D^\top,\mathfrak{T}(w,\mathbb{P})]_T$ is  $\delta$-observable.

  \item [$(iii)$]
      In this paper, we assume that the measurable space $(\Omega,\mathcal{F})$ is fixed solely for the convenience of statement. However, upon examining the proof of Theorem \ref{yu-theorem-7-24-1}, it becomes evident that this assumption can be dispensed with.
\end{enumerate}
\end{remark}
    \subsection{Approximate null controllability with cost and stabilizability}
    We now introduce the definitions of approximate null controllability with cost and feedback stabilizability for the system \eqref{yu-6-14-1}.

\begin{definition}\label{yu-def-7-26-1}
    Given  $\mathfrak{T}(w,\mathbb{P})$.
\begin{enumerate}
  \item[$(i)$]
      For a fixed $\delta\in [0,1)$ and $T>0$, the system
       $[A,B,C,D, \mathfrak{T}(w,\mathbb{P})]$ is deemed  $\delta$-null  controllable in
        $[0,T]$ with cost if there exists a constant $c(\delta,T)>0$  such that, for every
        $x_0\in \mathbb{R}^n$,  there exists a control
        $u\in L^2_{\mathfrak{T}(w,\mathbb{P})}(0,T;\mathbb{R}^m)$
         satisfying
\begin{equation}\label{yu-6-15-5-b}
    \|u\|_{L^2_{\mathfrak{T}(w,\mathbb{P})}(0,T;\mathbb{R}^m)}\leq c(\delta,T)|x_0|_{\mathbb{R}^n}\;\;\mbox{and}\;\;
    E_{\mathbb{P}}|x(T;u,x_0)|^2_{\mathbb{R}^n}\leq \delta|x_0|^2_{\mathbb{R}^n}.
\end{equation}
        In particular, if $\delta=0$, the system  $[A,B,C,D, \mathfrak{T}(w,\mathbb{P})]$  is referred to as null controllable in $[0,T]$. The constant $c(\delta,T)$ in \eqref{yu-6-15-5-b} is termed the cost for the  $\delta$-null controllability of $[A,B,C,D, \mathfrak{T}(w,\mathbb{P})]$ in $[0,T]$.

  \item[$(ii)$] The system $[A,B,C,D,\mathfrak{T}(w,\mathbb{P})]$ is said to be  approximately null controllable with cost if for any $\delta\in(0,1)$, there is $T:=T(\delta)>0$ such that  the system $[A,B,C,D,\mathfrak{T}(w,\mathbb{P})]$ is $\delta$-null controllable in $[0,T]$ with cost.
\end{enumerate}
\end{definition}
\begin{definition}\label{yu-defi-10-18-1}
    Given $\mathfrak{T}(w,\mathbb{P})$, the system $[A,B,C,D, \mathfrak{T}(w,\mathbb{P})]$ is
     said to be feedback stabilizable (or simply stabilizable) if
    there is an $F\in \mathbb{R}^{n\times m}$, $\delta>0$ and $C:=C(F,\delta,\mathfrak{T}(w,\mathbb{P}))>0$ such that, for every  $x_0\in\mathbb{R}^n$, the solution $x_F(\cdot;x_0)$ of the following closed-loop system:
\begin{equation}\label{yu-7-29-1}
\begin{cases}
    dx(t)=(A+BF)x(t)dt+\sum_{i=1}^d(C_i+D_iF)x(t)dw^i(t),\;\;t\in\mathbb{R}^+,\\
    x(0)=x_0
\end{cases}
\end{equation}
    satisfies  the condition:
\begin{equation}\label{yu-7-29-2}
    E_{\mathbb{P}}|x_F(t;x_0)|_{\mathbb{R}^n}^2\leq Ce^{-\delta t}|x_0|_{\mathbb{R}^n}\;\;\mbox{as}\;\;t\to+\infty.
\end{equation}
\end{definition}
\begin{remark}
    It is well-established in the theory of stability that there exist various definitions, such as exponential stability, $L^2$-stability. However, according to \cite[Proposition 3.6]{Huang-Li-Yong-2015} (or \cite[Lemma 2.3]{Sun-Yong-Zhang-2021}), these definitions are equivalent in the context of finite-dimensional stochastic systems..
\end{remark}

    According to \cite[Chapter 7, Theorem 7.17]{Lv-Zhang}, for a fixed
    $T>0$, \emph{$[A,B,C,D, \mathfrak{T}(w,\mathbb{P})]$     is null controllable on
    $[0,T]$ if and only if the system $[A^\top,B^\top,C^\top,D^\top, \mathfrak{T}(w,\mathbb{P})]_T$
      is initially observable on $[0,T]$}. Combining this with Theorem \ref{yu-theorem-7-24-1}, we can derive the following Corollary \ref{Co.gjia1}.

\begin{corollary}\label{Co.gjia1}
Given  $T>0$, the following two statements are equivalent:
    \begin{enumerate}
      \item [$(i)$] For any filtered probability space $\mathfrak{T}(w,\mathbb{P})$, the  system $[A,B,C,D, \mathfrak{T}(w,\mathbb{P})]$  is null controllable on $[0,T]$.

      \item [$(ii)$] There exists  a  filtered probability space $\mathfrak{T}(w,\mathbb{P})$ such that the  system $[A,B,C,D, \mathfrak{T}(w,\mathbb{P})]$  is null controllable on  $[0,T]$.
    \end{enumerate}
\end{corollary}

    Recently, in \cite{Trelat-Wang-Xu}, the authors introduced the concept of weak observability to investigate the exponential stabilizability of deterministic infinite-dimensional control systems. In that paper, they established the equivalence between weak observability, approximate null controllability with cost, and exponential stabilizability. \emph{It is noteworthy that in a stochastic framework, using weak observability (over a finite horizon) to characterize stabilizability (over an infinite horizon) is significantly more complex than the analogous problem for deterministic systems. Indeed, when studying the exponential stabilizability of stochastic control systems, the sample information varies across different equal-length intervals. This challenge motivates us to explore whether weak observability remains consistent across different probability spaces.}

Therefore, as another objective of this paper, \emph{we aim to extend the main result in \cite{Trelat-Wang-Xu} to stochastic control systems by leveraging Theorem \ref{yu-theorem-7-24-1}}.

    The second main result is Theorem \ref{yu-theorem-7-27-2} following.

\begin{theorem}\label{yu-theorem-7-27-2}
    The following statements are equivalent:
\begin{enumerate}
  \item [$(i)$]  For any filtered probability space $\mathfrak{T}(w,\mathbb{P})$, the system $[A,B,C,D, \mathfrak{T}(w,\mathbb{P})]$ is feedback stabilizable.

  \item [$(ii)$] For any filtered probability space $\mathfrak{T}(w,\mathbb{P})$, the system $[A,B,C,D, \mathfrak{T}(w,\mathbb{P})]$ is approximately  null controllable with  cost.

  \item [$(iii)$] There exists a filtered probability space $\mathfrak{T}(w,\mathbb{P})$ such that the system $[A,B,C,D, \mathfrak{T}(w,\mathbb{P})]$ is approximately  null controllable with  cost.

  \item [$(iv)$] There exists  a  filtered probability space $\mathfrak{T}(w,\mathbb{P})$ such that the system $[A^\top,B^\top,C^\top,D^\top,\mathfrak{T}(w,\mathbb{P})]$ is weakly observable.
\end{enumerate}
\end{theorem}

\begin{remark}\label{yu-remarl-8-20-1}
		 \begin{enumerate}
    \item[$(i)$]
        Theorem \ref{yu-theorem-7-27-2}, along with Remark \ref{yu-remark-7-26-1} in Section \ref{yu-sec-2-2}, reveals that the stabilizability and approximate null controllability with cost of the system \eqref{yu-6-14-1} are independent of the choice of the filtered probability space.

    \item[$(ii)$]
    Theorem \ref{yu-theorem-7-27-2} demonstrates that the approximate null controllability with cost of the system \eqref{yu-6-14-1} is strictly weaker than its null controllability. Specifically, when
    $C=0$ and $D=0$, the system \eqref{yu-6-14-1} is approximately null controllable with cost if and only if $\mbox{Rank}\,(\lambda I-A,B)=n$ for any $\lambda\in\{\eta\in\mathbb{C}:\mbox{Re}\,\eta\geq 0\}$.  However, this rank condition alone is insufficient to guarantee the null controllability of the system \eqref{yu-6-14-1}.

    \item[$(iii)$]
       According to \cite[Proposition 3.6]{Huang-Li-Yong-2015} (or \cite[Lemma 2.3]{Sun-Yong-Zhang-2021}) and Lemma \ref{yu-lemma-6-14-1}, one of the assertions in Theorem \ref{yu-theorem-7-27-2} holds if and only if the Riccati equation \eqref{yu-6-14-5} has a unique solution in
       $\mathbb{S}^n_+$. This provides an algebraic condition for the stabilizability and approximate null controllability with cost of the system \eqref{yu-6-14-1}.
      \end{enumerate}
\end{remark}

\section{The proof of Theorem \ref{yu-theorem-7-24-1}}\label{Se4}

\begin{proof}[\emph{{\bf Proof of Theorem \ref{yu-theorem-7-24-1}}}]
     Since $(ii)\Rightarrow (i)$ is evident, we only need to show $(i)\Rightarrow (ii)$. Let the filtered probability space $\mathfrak{T}(w,\mathbb{P})$ be fixed. Assume  that the system $[A^\top,B^\top,C^\top,D^\top,\mathfrak{T}(w,\mathbb{P})]_T$ is  $\delta$-observable on $[0,T]$.  Then, according to the definition of
     $\mathcal{O}_{op}:=\mathcal{O}_{op}(\delta,T,\mathfrak{T}(w,\mathbb{P}))$,   we obtain that,
     for any $y_1\in L^2_{\mathcal{F}_T^w(\Omega,\mathbb{P};\mathbb{R}^n)}$,
\begin{eqnarray}\label{yu-7-24-11}
    E_{\mathbb{P}}|y^w(0,y_1)|_{\mathbb{R}^n}^2
    \leq \mathcal{O}_{op}E_{\mathbb{P}}\int_0^T\Big|B^\top y^w(\sigma;y_1)+\sum_{i=1}^dD_i^\top Y^w_i(\sigma;y_1)\Big|^2_{\mathbb{R}^m}d\sigma+\delta E_{\mathbb{P}}|y_1|^2_{\mathbb{R}^n}.
\end{eqnarray}
      Next, consider an arbitrary filtered probability space $\mathfrak{T}(\widehat{w},\widehat{\mathbb{P}})$. We aim to prove that
\begin{equation}\label{yu-7-24-12}
    \mathcal{O}_{op}(\delta,T,\mathfrak{T}(\widehat{w},\widehat{\mathbb{P}}))\leq \mathcal{O}_{op}.
\end{equation}
       Upon establishing this inequality, the proof of $(i)\Rightarrow (ii)$   will be complete. Indeed, if  \eqref{yu-7-24-12} holds,  it is straightforward that  $[A^\top,B^\top,C^\top,D^\top,\mathfrak{T}(\widehat{w},\widehat{\mathbb{P}})]_T$ is  $\delta$-observable on $[0,T]$.
     Furthermore, by interchanging the roles of
    $\mathfrak{T}(\widehat{w},\widehat{\mathbb{P}})$ and $\mathfrak{T}(w,\mathbb{P})$ in the above arguments, we can conclude that
    $\mathcal{O}_{op}\leq \mathcal{O}_{op}(\delta,T,\mathfrak{T}(\widehat{w},\widehat{\mathbb{P}}))$.
     It follows that $\mathcal{O}_{op}= \mathcal{O}_{op}(\delta,T,\mathfrak{T}(\widehat{w},\widehat{\mathbb{P}}))$ and
    $\mathcal{O}_{op}$ is independent of the  the specific filtered probability space $\mathfrak{T}(w,\mathbb{P})$.

    We now proceed to prove \eqref{yu-7-24-12}. The proof is structured into several steps.

    {\bf Step 1: Fix $l\in \mathbb{N}^+$, $\alpha_k:=(\alpha_{k,1},\alpha_{k,2},\cdots,\alpha_{k,n})^\top\in\mathbb{R}^n$ and $h_k:=(h_{k,1},h_{k,2},\cdots,h_{k,n})\in (C([0,T];\mathbb{R}^d))^n$ for $k\in
    \{1,2,\ldots,l\}$  arbitrarily. Define
    $F_{l,N}$ according to   \eqref{yu-7-11-3} and set
    $\widehat{\xi}_{l,N}:=F_{l,N}(x^{\widehat{w}}(T))$, where $x^{\widehat{w}}$ is the solution of the equation on $\mathfrak{T}(\widehat{w},\widehat{\mathbb{P}})$:  \begin{equation}\label{yu-11-7-1}
dx^{\widehat{w}}(t)=-H^\top(t)dt+h^\top(t)d\widehat{w}(t)\;\;\mbox{for}
\;\;t\in(0,T)\;\;\mbox{with}\;\;x^{\widehat{w}}(0)=0.
\end{equation}
   Here,  $H$ and $h$ are  given by \eqref{yu-7-23-b0-001}. We aim to prove that
\begin{equation}\label{yu-7-24-14}
    E_{\widehat{\mathbb{P}}}|y^{\widehat{w}}(0;\widehat{\xi}_{l,N})|_{\mathbb{R}^n}^2
    \leq \mathcal{O}_{op}E_{\widehat{\mathbb{P}}}\int_0^T\Big|B^\top y^{\widehat{w}}(\sigma;\widehat{\xi}_{l,N})+\sum_{i=1}^dD_i^\top Y^{\widehat{w}}_i(\sigma;\widehat{\xi}_{l,N})\Big|^2_{\mathbb{R}^m}
    d\sigma+\delta E_{\widehat{\mathbb{P}}}|\widehat{\xi}_{l,N}|^2_{\mathbb{R}^n}.
\end{equation}}
    Let $x^{w}$ be the solution  of the following equation on $\mathfrak{T}(w,\mathbb{P})$:
\begin{equation*}\label{yu-7-24-13}
\begin{cases}
    dx^w(t)=-H^\top(t)dt+h^\top(t)dw(t),&t\in(0,T),\\
    x^w(0)=0,
\end{cases}
\end{equation*}
    and set $\xi_{l,N}:=F_{l,N}(x^w(T))$.
    According to Lemma \ref{yu-proposition-7-11-1}, we have
\begin{equation}\label{yu-7-24-15}
    y^{\widehat{w}}(\cdot;\widehat{\xi}_{l,N})
    =U(\cdot,x^{\widehat{w}}(\cdot))\;\;\mbox{and}
    \;\;Y^{\widehat{w}}(\cdot;\widehat{\xi}_{l,N})=\nabla U(\cdot,x^{\widehat{w}}(\cdot))h^\top(\cdot)
\end{equation}
and
\begin{equation}\label{yu-7-24-17}
     y^{w}(\cdot;\xi_{l,N})=U(\cdot,x^w(\cdot))\;\;\mbox{and}
    \;\;Y^{w}(\cdot;\xi_{l,N})=\nabla U(\cdot,x^w(\cdot))h^\top(\cdot)
\end{equation}
    in $[0,T]$, where $U\in C^{1,2}([0,T]\times \mathbb{R}^{nl};\mathbb{R}^n)$ is the unique solution to  the equation \eqref{yu-7-22-4}. By the Yamada-Watanabe theorem (see \cite{Yamada-Watanabe-1971}), we can conclude that
    $x^w$ and $x^{\widehat{w}}$ have the same probability law in $\mathbb{R}^+$.  This, along with \eqref{yu-7-24-11}, \eqref{yu-7-24-15}, and \eqref{yu-7-24-17}, implies that
\begin{eqnarray*}\label{yu-7-24-18}
    &\;&E_{\mathbb{P}}|y^{\widehat{w}}(0;\widehat{\xi}_{l,N})|^2_{\mathbb{R}^n}
    =|U(0,0)|_{\mathbb{R}^n}^2= E_{\mathbb{P}}|y^{w}(0;\xi_{l,N})|^2_{\mathbb{R}^n}\nonumber\\
    &\leq&\mathcal{O}_{op}E_{\mathbb{P}}\int_0^T\Big|B^\top y^w(\sigma;\xi_{l,N})+\sum_{i=1}^dD_i^\top Y^w_i(\sigma;\xi_{l,N})\Big|^2_{\mathbb{R}^m}d\sigma+\delta E_{\mathbb{P}}|\xi_{l,N}|^2_{\mathbb{R}^n}\nonumber\\
    &=&\mathcal{O}_{op}E_{\mathbb{P}}\int_0^T\Big|B^\top U(\sigma,x^w(\sigma))+\sum_{i=1}^dD_i^\top (\nabla U(\sigma,x^w(\sigma))h^\top(\sigma))_i\Big|^2_{\mathbb{R}^m}d\sigma
    +\delta E_{\mathbb{P}}|F_{l,N}(x^w(T))|^2_{\mathbb{R}^n}\nonumber\\
    &=&\mathcal{O}_{op}E_{\widehat{\mathbb{P}}}\int_0^T\Big|B^\top U(\sigma,x^{\widehat{w}}(\sigma))+\sum_{i=1}^dD_i^\top (\nabla U(\sigma,x^{\widehat{w}}(\sigma))h^\top(\sigma))_i\Big|^2_{\mathbb{R}^m}d\sigma
    +\delta E_{\widehat{\mathbb{P}}}|F_{l,N}(x^{\widehat{w}}(T))|^2_{\mathbb{R}^n}\nonumber\\
    &=&\mathcal{O}_{op}E_{\widehat{\mathbb{P}}}\int_0^T\Big|B^\top y^{\widehat{w}}(\sigma;\widehat{\xi}_{l,N})+\sum_{i=1}^dD_i^\top Y^{\widehat{w}}_i(\sigma;\widehat{\xi}_{l,N})\Big|^2_{\mathbb{R}^m}d\sigma
    +\delta E_{\widehat{\mathbb{P}}}|\widehat{\xi}_{l,N}|^2_{\mathbb{R}^n}.
\end{eqnarray*}
      Thus, \eqref{yu-7-24-14} holds.

{ \bf Step 2:  We prove that the claim in Step 1 also holds when
 $h_k:=(h_{k,1},h_{k,2},\cdots,h_{k,n})\in (L^2(0,T;\mathbb{R}^d))^n$ for $k\in
    \{1,2,\ldots,l\}$.}
     We first recall that there exists a constant  $c(T)>0$ such that  (see, for instance, \cite[Chapter 4, Theorem 4.2]{Lv-Zhang})
\begin{eqnarray}\label{yu-7-25-1}
    &\;&E_{\widehat{\mathbb{P}}}\Big(\sup_{0\leq t\leq T}|y^{\widehat{w}}(t;\xi_1)-y^{\widehat{w}}(t;\xi_2)|_{\mathbb{R}^n}^2\Big)
    +\sum_{i=1}^d E_{\widehat{\mathbb{P}}}\Big(\int_0^T|Y^{\widehat{w}}_i(t;\xi_1)
    -Y^{\widehat{w}}_i(t;\xi_2)|^2_{\mathbb{R}^n}dt
    \Big)\nonumber\\
    &=&E_{\widehat{\mathbb{P}}}\Big(\sup_{0\leq t\leq T}|y^{\widehat{w}}(t;\xi_1-\xi_2)|_{\mathbb{R}^n}^2\Big)
    +\sum_{i=1}^d E_{\widehat{\mathbb{P}}}\Big(\int_0^T|Y^{\widehat{w}}_i
    (t;\xi_1-\xi_2)|^2_{\mathbb{R}^n}dt
    \Big)\nonumber\\
    &\leq& c(T)E_{\widehat{\mathbb{P}}} |\xi_1-\xi_2|_{\mathbb{R}^n}^2
    \;\;\mbox{for any}\;\;\xi_1,\xi_2\in L^2_{\mathcal{F}_T^{\widehat{w}}}(\Omega,\widehat{\mathbb{P}};\mathbb{R}^n).
\end{eqnarray}
      We now choose an arbitrary $\varepsilon>0$. It is well known that  there exists  $h^\varepsilon_k:=(h^\varepsilon_{k,1},h^\varepsilon_{k,2},\cdots,
    h^\varepsilon_{k,n})\in (C([0,T];\mathbb{R}^d))^n$ for $k\in
    \{1,2,\ldots,l\}$ such that
\begin{equation}\label{yu-7-25-2}
    \int_0^T\big(|H^\varepsilon(t)-H(t)|^2_{\mathbb{R}^{nl}}
    +|h^\varepsilon(t)-h(t)|^2_{\mathbb{R}^{d\times nl}}\big)dt<\varepsilon.
\end{equation}
    where
\begin{equation}\label{yu-7-25-2-bbb}
h^\varepsilon:=(h^\varepsilon_1,h^\varepsilon_2,\cdots,h^\varepsilon_l)\;\;\mbox{and}\;\; H^\varepsilon:=(H^\varepsilon_1,H^\varepsilon_2,
    \cdots,H^\varepsilon_l)
\end{equation}
    with $H_k^\varepsilon:=\frac{1}{2}(|h^\varepsilon_{k,1}|^2_{\mathbb{R}^d}, |h^\varepsilon_{k,2}|^2_{\mathbb{R}^d},
    \cdots,|h^\varepsilon_{k,n}|^2_{\mathbb{R}^d})$.
    Let $x^{\widehat{w}}_\varepsilon$ be the solution of the equation
\begin{equation}\label{yu-7-25-3}
    \begin{cases}
    dx^{\widehat{w}}_\varepsilon(t)=-(H^\varepsilon(t))^\top dt+(h^\varepsilon(t))^\top d\widehat{w}(t),&t\in(0,T),\\
    x^{\widehat{w}}_\varepsilon(0)=0.
\end{cases}
\end{equation}
    Using  the H\"{o}lder inequality and \eqref{yu-7-25-2},  it is straightforward to verify that
$E_{\widehat{\mathbb{P}}}|x_\varepsilon^{\widehat{w}}(T)
-x^{\widehat{w}}(T)|_{\mathbb{R}^{nl}}^2
\leq 2(1+T)\varepsilon$.
     This, combined with the definition of  $F_{l,N}$, yields
\begin{equation}\label{yu-7-25-5}
    E_{\widehat{\mathbb{P}}}
    |F_{l,N}(x_\varepsilon^{\widehat{w}}(T))
    -F_{l,N}(x^{\widehat{w}}(T))|^2_{\mathbb{R}^n}
    \leq c(N,l,F)E_{\widehat{\mathbb{P}}}|x_\varepsilon^{\widehat{w}}(T)
-x^{\widehat{w}}(T)|_{\mathbb{R}^{nl}}^2
\leq c(N,l,F,T)\varepsilon.
\end{equation}
     (Note here that  $F_{l,N}(\xi_1)-F_{l,N}(\xi_2)
=\int_0^1\nabla F_{l,N}(\xi_2
+\rho(\xi_1-\xi_2))d\rho
(\xi_1-\xi_2)$ for any $\xi_1,\xi_2\in\mathbb{R}^{nl}$.) Let $\widehat{\xi}_{l,N}^\varepsilon:=F_{l,N}(x^{\widehat{w}}_\varepsilon(T))$ and $\widehat{\xi}_{l,N}:=F_{l,N}(x^{\widehat{w}}(T))$. Then, by \eqref{yu-7-25-1} and \eqref{yu-7-25-5}, we can readily deduce that
\begin{eqnarray}\label{yu-7-25-6}
     &\;&E_{\widehat{\mathbb{P}}}|y^{\widehat{w}}(0;\widehat{\xi}_{l,N}^\varepsilon
     -\widehat{\xi}_{l,N})|^2_{\mathbb{R}^n}\nonumber\\
     &\;&+E_{\widehat{\mathbb{P}}}\int_0^T\Big|B^\top y^{\widehat{w}}(\sigma;\widehat{\xi}_{l,N}^\varepsilon-
     \widehat{\xi}_{l,N})+\sum_{i=1}^dD_i^\top Y_i^{\widehat{w}}(\sigma;\widehat{\xi}_{l,N}^\varepsilon-\widehat{\xi}_{l,N})
     \Big|^2_{\mathbb{R}^m}d\sigma\leq c(N,l,F,T)\varepsilon.
\end{eqnarray}
     Utilizing \eqref{yu-7-25-1} and the definitions of $F_{l,N}$, $\widehat{\xi}_{l,N}$  and $\widehat{\xi}_{l,N}^\varepsilon$, we also can deduce  that
\begin{equation}\label{yu-7-25-7}
    E_{\widehat{\mathbb{P}}}|y^{\widehat{w}}(0;\widehat{\xi}_{l,N}^\varepsilon)|^2_{\mathbb{R}^n}
    +E_{\widehat{\mathbb{P}}}\int_0^T\Big|B^\top y^{\widehat{w}}(\sigma;\widehat{\xi}_{l,N})+\sum_{i=1}^dD_i^\top Y_i^{\widehat{w}}(\sigma;\widehat{\xi}_{l,N})
     \Big|^2_{\mathbb{R}^m}d\sigma\leq c(N,F,T).
\end{equation}
     Therefore, by \eqref{yu-7-25-6}, \eqref{yu-7-25-7}, and the H\"{o}lder inequality, we obtain
\begin{equation}\label{yu-7-25-8}
    E_{\widehat{\mathbb{P}}}|y^{\widehat{w}}(0;\widehat{\xi}_{l,N})|^2_{\mathbb{R}^n}\leq
   E_{\widehat{\mathbb{P}}}|y^{\widehat{w}}(0;\widehat{\xi}_{l,N}^\varepsilon)|^2_{\mathbb{R}^n}
    +c(N,l,F,T)(\sqrt{\varepsilon}+\varepsilon)
\end{equation}
    and
\begin{eqnarray}\label{yu-7-25-9}
    &\;&E_{\widehat{\mathbb{P}}}\int_0^T\Big|B^\top y^{\widehat{w}}(\sigma;\widehat{\xi}^\varepsilon_{l,N})+\sum_{i=1}^dD_i^\top Y_i^{\widehat{w}}(\sigma;\widehat{\xi}^\varepsilon_{l,N})
     \Big|^2_{\mathbb{R}^m}d\sigma\nonumber\\
     &\leq& E_{\widehat{\mathbb{P}}}\int_0^T\Big|B^\top y^{\widehat{w}}(\sigma;\widehat{\xi}_{l,N})+\sum_{i=1}^dD_i^\top Y_i^{\widehat{w}}(\sigma;\widehat{\xi}_{l,N})
     \Big|^2_{\mathbb{R}^m}d\sigma+c(N,l,F,T)(\sqrt{\varepsilon}+\varepsilon).
\end{eqnarray}
     Combining these results with the claim in {\bf Step 1} and \eqref{yu-7-25-5}, we can conclude that
\begin{eqnarray*}\label{yu-7-25-10}
    E_{\widehat{\mathbb{P}}}
    |y^{\widehat{w}}(0;\widehat{\xi}_{l,N})|^2_{\mathbb{R}^n}\leq
    \mathcal{O}_{op}E_{\widehat{\mathbb{P}}}\int_0^T\Big|B^\top y^{\widehat{w}}(\sigma;\widehat{\xi}_{l,N})+\sum_{i=1}^dD_i^\top Y_i^{\widehat{w}}(\sigma;\widehat{\xi}_{l,N})
     \Big|^2_{\mathbb{R}^m}d\sigma\nonumber\\
     + \delta E_{\widehat{\mathbb{P}}}|\widehat{\xi}_{l,N}|^2_{\mathbb{R}^n}
     +c(N,F,T,\mathcal{O}_{op},\delta)(\sqrt{\varepsilon}+\varepsilon).
\end{eqnarray*}
      Due to the arbitrariness of  $\varepsilon$ we can conclude that \eqref{yu-7-24-14} holds when $h_k:=(h_{k,j})_{j=1}^n\in (L^2(0,T;\mathbb{R}^d))^n$ for $k\in
    \{1,2,\ldots,l\}$.

    {\bf Step 3:  Let $l\in \mathbb{N}^+$, $\alpha_k:=(\alpha_{k,1},\alpha_{k,2},\cdots,\alpha_{k,n})^\top\in\mathbb{R}^n$ and $h_k:=(h_{k,1},h_{k,2},\cdots,h_{k,n})\in (L^2(0,T;\mathbb{R}^d))^n$ for $k\in
    \{1,2,\ldots,l\}$ be fixed arbitrarily. Define
    $F_l$   by \eqref{yu-7-23-bbbb-1} and let
    $\widehat{\xi}_{l}:=F_l(x^{\widehat{w}}(T))$, where $x^{\widehat{w}}$ is the solution of the equation \eqref{yu-11-7-1} on $\mathfrak{T}(\widehat{w},\widehat{\mathbb{P}})$. We will demonstrate that \eqref{yu-7-24-14} holds with
       $\widehat{\xi}_{l,N}$ replaced  by $\widehat{\xi}_{l}$.}
         By the definition of $F_l$,  it is straightforward to show that
$F_l(x^{\widehat{w}}(T))\in L^2_{\mathcal{F}^{\widehat{w}}_T}(\Omega,\widehat{\mathbb{P}};\mathbb{R}^n)$.
 Furthermore, by applying the dominated convergence theorem, we obtain
\begin{equation*}\label{yu-7-11-4}
    \lim_{N\to+\infty}E_{\widehat{\mathbb{P}}}|\widehat{\xi}_{l,N}
    -\widehat{\xi}_{l}|^2_{\mathbb{R}^n}=
   \lim_{N\to+\infty}
    E_{\widehat{\mathbb{P}}}|F_{l,N}(x^{\widehat{w}}(T))
    -F_l(x^{\widehat{w}}(T))|^2_{\mathbb{R}^n}
    =0.
\end{equation*}
     Hence, using \eqref{yu-7-25-1}, we can deduce
\begin{equation*}\label{yu-7-25-12}
    \lim_{N\to+\infty}\Big[
    E_{\widehat{\mathbb{P}}}\Big(\sup_{0\leq t\leq T}|y^{\widehat{w}}(t;\widehat{\xi}_{l,N})
    -y^{\widehat{w}}(t;\widehat{\xi}_{l})|_{\mathbb{R}^n}^2\Big)
    +\sum_{i=1}^d E_{\widehat{\mathbb{P}}}\Big(\int_0^T|Y^{\widehat{w}}_i(t;\widehat{\xi}_{l,N})
    -Y^{\widehat{w}}_i(t;\widehat{\xi}_{l})|^2_{\mathbb{R}^n}dt
    \Big)\Big]=0.
\end{equation*}
    This, along with the claim in \emph{Step 2}, yields that \eqref{yu-7-24-14} holds with replacing $\widehat{\xi}_{l,N}$ by $\widehat{\xi}_{l}$.

 {\bf Step 4: Completion of  the proof.}
    To complete the proof, we remind ourselves that the set
 \begin{equation*}\label{yu-7-25-13}
    \mathcal{X}:=\mbox{span}\Big\{\exp\Big\{\int_0^Th^\top (t)d\widehat{w}(t)
    -\frac{1}{2}\int_0^T|h(t)|^2_{\mathbb{R}^d}dt\Big\}:h\in L^2(0,T;\mathbb{R}^d)\Big\}
\end{equation*}
    is dense in $L^2_{\mathcal{F}^{\widehat{w}}_T}(\Omega,\widehat{\mathbb{P}};\mathbb{R})$ (refer, for example,  \cite[Chapter 2, Lemma 2.146]{Lv-Zhang}). Let $\varepsilon>0$ and $\xi\in  L^2_{\mathcal{F}^{\widehat{w}}_T}(\Omega,\widehat{\mathbb{P}};\mathbb{R}^n)$ with $E_{\widehat{\mathbb{P}}}|\xi|^2_{\mathbb{R}^n}=1$ be fixed arbitrarily.   Given the density of  $\mathcal{X}$ in $L^2_{\mathcal{F}^{\widehat{w}}_T}(\Omega,\widehat{\mathbb{P}};\mathbb{R})$,
     we can deduce the existence of   $l:=l(\varepsilon)\in  \mathbb{N}^+$,   $\{\alpha^{\varepsilon}_k\}_{k=1}^l$ and $\{h^{\varepsilon}_k\}_{k=1}^l$ with  $\alpha^{\varepsilon}_k:=(\alpha^{\varepsilon}_{k,1},\alpha^{\varepsilon}_{k,2},
    \cdots,\alpha^{\varepsilon}_{k,n})^\top\in\mathbb{R}^n$ and $h^{\varepsilon}_k:=(h^{\varepsilon}_{k,1},h^{\varepsilon}_{k,2},
    \cdots,h^{\varepsilon}_{k,n})\in (L^2(0,T;\mathbb{R}^d))^n$   such that
$E_{\widehat{\mathbb{P}}}\big|\sum_{k=1}^lG^\varepsilon
    _k(\widehat{w})-\xi\big|^2_{\mathbb{R}^n}\leq \varepsilon$,
where
$G^\varepsilon_k(\widehat{w}):=(G^\varepsilon_{k,1}(\widehat{w}),$
   $G^\varepsilon_{k,2}(\widehat{w}),
   \cdots,G^\varepsilon_{k,n}(\widehat{w}))^\top$
    with
$G^\varepsilon_{k,j}(\widehat{w}):=\alpha_{k,j}\exp\big(\int_0^T
    (h^\varepsilon_{k,j}(s))^\top
    d\widehat{w}(s)-\frac{1}{2}\int_0^T
     |h^\varepsilon_{k,j}(s)|^2_{\mathbb{R}^d}ds\big)$
     for each $j=1,2,\ldots,n$.
 For this $\{\alpha_k\}_{k=1}^l$, we define $F_l$  according to \eqref{yu-7-23-bbbb-1}. It is evident that $
\sum_{k=1}^lG^\varepsilon_k(\widehat{w})=F_l(x_\varepsilon^{\widehat{w}}(T)),
$
    where $x_\varepsilon^{\widehat{w}}(\cdot)$ is the solution of the equation
    \eqref{yu-7-25-3} with  $H^\varepsilon(\cdot)$ and $h^\varepsilon(\cdot)$  defined by \eqref{yu-7-25-2-bbb}. We let $\xi_\varepsilon:=F_l(x_\varepsilon^{\widehat{w}}(T))$.
    Utilizing \eqref{yu-7-25-1} and employing a similar approach to the one used in proving \eqref{yu-7-25-8} and \eqref{yu-7-25-9}, we can readily demonstrate that
$|y^{\widehat{w}}(0;\xi)|^2_{\mathbb{R}^n}\leq
    |y^{\widehat{w}}(0;\xi_\varepsilon)|^2_{\mathbb{R}^n}
    +c(T)(\sqrt{\varepsilon}+\varepsilon)$
    and
\begin{eqnarray*}\label{yu-7-26-3}
    E_{\widehat{\mathbb{P}}}\int_0^T\Big|B^\top y^{\widehat{w}}(\sigma;\xi_\varepsilon)+\sum_{i=1}^dD_i^\top Y_i^{\widehat{w}}(\sigma;\xi_\varepsilon)
     \Big|^2_{\mathbb{R}^m}d\sigma
     \leq E_{\widehat{\mathbb{P}}}\int_0^T\Big|B^\top y^{\widehat{w}}(\sigma;\xi)+\sum_{i=1}^dD_i^\top Y_i^{\widehat{w}}(\sigma;\xi)
     \Big|^2_{\mathbb{R}^m}d\sigma\nonumber\\
     +c(T)(\sqrt{\varepsilon}+\varepsilon).
\end{eqnarray*}
     Combining these observations with the assertion made in {\bf Step 3}, we infer that
\begin{eqnarray*}\label{yu-7-26-4}
    E_{\widehat{\mathbb{P}}}|y^{\widehat{w}}(0;\xi)|^2_{\mathbb{R}^n}\leq\mathcal{O}_{op}
    E_{\widehat{\mathbb{P}}}\int_0^T\Big|B^\top y^{\widehat{w}}(\sigma;\xi)+\sum_{i=1}^dD_i^\top Y_i^{\widehat{w}}(\sigma;\xi)
     \Big|^2_{\mathbb{R}^m}d\sigma+\delta E_{\widehat{\mathbb{P}}}|\xi|^2_{\mathbb{R}^n}
     +c(T,\mathcal{O}_{op},\delta)(\sqrt{\varepsilon}+\varepsilon).
\end{eqnarray*}
    This, coupled with the second-order homogeneity (in terms of $\xi$) of weak observability and the arbitrariness of $\varepsilon$ and $\xi$, leads to the conclusion that
\begin{equation*}
    E_{\widehat{\mathbb{P}}}|y^{\widehat{w}}(0;\xi)|^2_{\mathbb{R}^n}\leq\mathcal{O}_{op}
    E_{\widehat{\mathbb{P}}}\int_0^T\Big|B^\top y^{\widehat{w}}(\sigma;\xi)+\sum_{i=1}^dD_i^\top Y_i^{\widehat{w}}(\sigma;\xi)
     \Big|^2_{\mathbb{R}^m}d\sigma+\delta E_{\widehat{\mathbb{P}}}|\xi|^2_{\mathbb{R}^n}\;\;\mbox{for any}\;\;\xi\in L^2_{\mathcal{F}^{\widehat{w}}_T}(\Omega,\widehat{\mathbb{P}};\mathbb{R}^n).
\end{equation*}
     Consequently, \eqref{yu-7-24-12} holds. Thus, the proof is complete.
\end{proof}
\section{The proof of Theorem \ref{yu-theorem-7-27-2}}\label{Se5}

    Before presenting the proof of Theorem \ref{yu-theorem-7-27-2}, we first provide some characterizations of approximate null controllability with cost in a general time interval.
\begin{definition}\label{yu-def-6-15-1}
    Given $\mathfrak{T}(w,\mathbb{P})$, $\delta\in (0,1)$ and $0<s<T$. The system $[A,B,C,D, \mathfrak{T}(w,\mathbb{P})]$ is said to be  $\delta$-null controllable in  $[s,T]$ with cost if  there exists a constant $c(\delta,s,T)>0$ such that, for any  $x_s\in L^2_{\mathcal{F}^w_s}(\Omega,\mathbb{P};\mathbb{R}^n)$,   there exists a control  $u\in L^2_{\mathfrak{T}(w,\mathbb{P})}(s,T;\mathbb{R}^m)$  satisfying
\begin{equation}\label{yu-6-15-5}
    \|u\|_{L^2_{\mathfrak{T}(w,\mathbb{P})}(s,T;\mathbb{R}^m)}\leq c(\delta,s,T)\|x_s\|_{L^2_{\mathcal{F}^w_s}(\Omega,\mathbb{P};\mathbb{R}^n)}
    \;\;\mbox{and}\;\;
    E_{\mathbb{P}}|x(T;u,x_s,s)|^2_{\mathbb{R}^n}\leq \delta\|x_s\|^2_{L^2_{\mathcal{F}^w_s}(\Omega,\mathbb{P};\mathbb{R}^n)},
\end{equation}
 where  and in what follows  $x(\cdot;u,x_s,s)$ is the solution to  the following equation
\begin{equation}\label{yu-7-3-1}
\begin{cases}
      dx(t)=(Ax(t)+Bu(t))dt+\sum_{i=1}^d(C_ix(t)+D_iu(t))dw^i(t),\;\;t\in(s,T),\\
      x(s)=x_s.
\end{cases}
\end{equation}
    The constant  $c(\delta,s,T)$  is referred to as the cost for the   $\delta$-null controllability of $[A,B,C,D, \mathfrak{T}(w,\mathbb{P})]$  in $[s,T]$.
\end{definition}

     The following theorem holds independent significance within the realm of control theory.

\begin{theorem}\label{yu-theorem-6-22-1}
    Let  $0<s<T$ be fixed. Then the following statements are equivalent:
\begin{enumerate}
  \item [$(i)$] The system $[A,B,C,D,\mathfrak{T}(w,\mathbb{P})]$ is  $\delta$-null controllable in $[s,T]$ with cost for some $\delta\in(0,1)$.

  \item [$(ii)$] There exists  $\delta\in(0,1)$ and $c>0$  such that, for any $y_1\in L^2_{\mathcal{F}^w_T}(\Omega,\mathbb{P};\mathbb{R}^n)$,
\begin{equation}\label{yu-7-30-1}
     E_{\mathbb{P}}|y(s;y_1)|^2_{\mathbb{R}^n}
    \leq c E_{\mathbb{P}}\int_s^T\Big|B^\top y(t;y_1)+\sum_{i=1}^dD_i^\top Y_i(t;y_1)\Big|^2_{\mathbb{R}^m}dt
    +\delta E_{\mathbb{P}}|y_1|_{\mathbb{R}^n}^2,
\end{equation}
where $(y(\cdot;y_1), Y(\cdot;y_1):=(Y_i(\cdot;y_1))_{i=1}^d)$  is the solution to the BSDE in \eqref{yu-7-26-10} with terminal value  $y_1\in L^2_{\mathcal{F}^w_T}(\Omega,\mathbb{P};\mathbb{R}^n)$.
\end{enumerate}
    In particular, if $(i)$ holds with cost $c(\delta,s,T)$, then \eqref{yu-7-30-1}   is satisfied by  substituting $(c,\delta)$ with  $(c(\delta,s,T)(1+2(1-\delta)^{-1}),\frac{1+\delta}{2})$. Conversely,
    if
    \eqref{yu-7-30-1} holds for any $y_1\in L^2_{\mathcal{F}^w_T}(\Omega,\mathbb{P};\mathbb{R}^n)$, then $[A,B,C,D,\mathfrak{T}(w,\mathbb{P})]$ is $\delta$-null controllable on $[s,T]$ with cost $\sqrt{c\delta^{-1}c_0(T-s)}$. Here,   $c_0(T-s)$ is the constant  such that
\begin{equation}\label{yu-b-8-8-1}
    E_{\mathbb{P}}|x(T;0,\xi,s)|_{\mathbb{R}^n}^2\leq c_0(T-s)E_{\mathbb{P}}|\xi|_{\mathbb{R}^n}^2\;\; \mbox{for any}\;\;\xi \in L^2_{\mathcal{F}^w_s}
    (\Omega,\mathbb{P};\mathbb{R}^n).
\end{equation}

\end{theorem}
 \begin{remark}\label{yu-remark-4-2}
 \begin{enumerate}
   \item [$(i)$] By utilizing the classical variation of constants formula for SDEs and the Gronwall inequality, it can be readily verified that there exists a function $c_0:\mathbb{R}^+\to\mathbb{R}^+$ such  that, for any $t_1,t_2\in [0,+\infty)$ with $ t_2>t_1$,
     it holds that
        $E_{\mathbb{P}}|x(t_2;0,\xi,t_1)|_{\mathbb{R}^n}^2\leq c_0(t_2-t_1)E_{\mathbb{P}}|\xi|_{\mathbb{R}^n}^2$ for any
        $\xi \in L^2_{\mathcal{F}^w_{t_1}}
    (\Omega,\mathbb{P};\mathbb{R}^n)$.
   \item [$(ii)$] In Theorem \ref{yu-theorem-6-22-1}, we have presented an explicitly relationship between the cost $c(\delta,s,T)$ of $\delta$-null controllability of $[A,B,C,D,\mathfrak{T}(w,\mathbb{P})]$ and observability constant $c$ in  \eqref{yu-7-30-1}. Whether they are equal or not is still unknown to us.
 \end{enumerate}

\end{remark}
\begin{proof}[\emph{{\bf Proof of  Theorem \ref{yu-theorem-6-22-1}}}]
     We first prove the implication $(i)\Rightarrow(ii)$. Suppose that $\delta\in (0,1)$  such  that the system $[A,B,C,D,\mathfrak{T}(w,\mathbb{P})]$ is  $\delta$-null controllable in $[s,T]$ with cost.
      Consequently, there exists a constant  $c(\delta,s,T)>0$ such that, for any $x_s$, there
       exists  a control $u\in  L^2_{\mathfrak{T}(w,\mathbb{P})}(s,T;\mathbb{R}^m)$   satisfying
       the inequalities in \eqref{yu-6-15-5}.

     Fix $x_s$ arbitrarily  and choose  $u\in  L^2_{\mathfrak{T}(w,\mathbb{P})}(s,T;\mathbb{R}^m)$ be such that  \eqref{yu-6-15-5} holds. For any  $y_1$, we apply    It\^{o}'s formula to
      the inner product $\langle x(\cdot;u,x_s,s),y(\cdot;y_1)\rangle_{\mathbb{R}^n}$
      (over the interval  $[s,T]$),  yielding
\begin{eqnarray*}\label{yu-6-22-10}
    &\;&d\langle x(t;u,x_s,s),y(t;y_1)\rangle_{\mathbb{R}^n}
    =\Big\langle u(t),B^\top y(t;y_1)+\sum_{i=1}^dD_i^\top Y_i(t;y_1)\Big\rangle_{\mathbb{R}^m}dt\nonumber\\
    &\;&+\sum_{i=1}^d[\langle C_ix(t;u,x_s,s)+ D_iu(t),y(t;y_1)\rangle_{\mathbb{R}^n}
    +\langle x(t;u,x_s,s),Y_i(t;y_1)\rangle_{\mathbb{R}^n}]dw^i(t).
\end{eqnarray*}
     As a result, we obtain the following equation
\begin{equation}\label{yu-6-22-11}
    E_{\mathbb{P}}\langle x(T;u,x_s,s),y_1\rangle_{\mathbb{R}^n}-E_{\mathbb{P}}\langle x_s,y(s;y_1)\rangle_{\mathbb{R}^n}
    =E_{\mathbb{P}}\int_s^T\Big\langle u(t),B^\top y(t;y_1)+\sum_{i=1}^dD_i^\top Y_i(t;y_1)\Big\rangle_{\mathbb{R}^m}dt.
\end{equation}
     Combining this equation with \eqref{yu-6-15-5} and applying the H\"{o}lder inequality, we derive
\begin{eqnarray*}\label{yu-6-22-12}
    &\;&|E_{\mathbb{P}}\langle x_s,y(s;y_1)\rangle_{\mathbb{R}^n}|\nonumber\\
    &\leq&E_{\mathbb{P}}\int_s^T\Big(|u(t)|_{\mathbb{R}^m}\Big|B^\top y(t;y_1)+\sum_{i=1}^dD_i^\top Y_i(t;y_1)\Big|_{\mathbb{R}^m}\Big)dt +E_{\mathbb{P}}(|y_1|_{\mathbb{R}^n}|x(T;u,x_s,s)|_{\mathbb{R}^n})
    \nonumber\\
    &\leq&\sqrt{c(\delta, s, T)}\big(E_{\mathbb{P}}|x_s|^2_{\mathbb{R}^n}\big)^{\frac{1}{2}}
    \biggl(E_{\mathbb{P}}\int_s^T\Big|B^\top y(t;y_1)+\sum_{i=1}^dD_i^\top Y_i(t;y_1)\biggl|^2_{\mathbb{R}^m}dt\Big)^{\frac{1}{2}}
    +\sqrt{\delta}(E_{\mathbb{P}}|y_1|_{\mathbb{R}^n}^2)^{\frac{1}{2}}
   \Big(E_{\mathbb{P}}|x_s|^2_{\mathbb{R}^n}\Big)^{\frac{1}{2}}.
\end{eqnarray*}
     Therefore, due to the arbitrariness of  $x_s$, we can conclude
\begin{equation*}\label{yu-6-23-1}
     E_{\mathbb{P}}|y(s;y_1)|^2_{\mathbb{R}^n}
    \leq c(\delta,s,T)(1+2(1-\delta)^{-1})E_{\mathbb{P}}\int_s^T\Big|B^\top y(t;y_1)+\sum_{i=1}^dD_i^\top Y_i(t;y_1)\Big|^2_{\mathbb{R}^m}dt
    +\frac{1+\delta}{2} E_{\mathbb{P}}|y_1|_{\mathbb{R}^n}^2.
\end{equation*}
     Given the arbitrariness of $y_1$ and the condition $\delta\in(0,1)$,   we can deduce that \eqref{yu-7-30-1} holds with the substitution of  $(c,\delta)$ by $(c(\delta,s,T)(1+2(1-\delta)^{-1}),\frac{1+\delta}{2})$.  Consequently, $(ii)$ is satisfied.

      We will now prove the implication  $(ii)\Rightarrow (i)$. Let $\delta\in (0,1)$  be given, and suppose that there exists a constant  $c>0$ such that \eqref{yu-7-30-1} holds for any $y_1\in L^2_{\mathcal{F}^w_T}(\Omega,\mathbb{P};\mathbb{R}^n)$.
        We define the mapping $\mathcal{G}_{[s,T]}^\delta: L^2_{\mathcal{F}^w_T}(\Omega,\mathbb{P};\mathbb{R}^n)
        \to  L^2_{\mathcal{F}^w_T}(\Omega,\mathbb{P};\mathbb{R}^n)$ as follows:
\begin{eqnarray}\label{yu-6-23-2}
     \langle \mathcal{G}_{[s,T]}^\delta \xi,\eta\rangle_{L^2_{\mathcal{F}^w_T}(\Omega,\mathbb{P};\mathbb{R}^n)}
     :=cE_{\mathbb{P}}\int_s^T\Big\langle B^\top y(t;\xi)+\sum_{i=1}^dD_i^\top Y_i(t;\xi), B^\top y(t;\eta)+\sum_{i=1}^dD_i^\top Y_i(t;\eta)\Big\rangle_{\mathbb{R}^m}dt\nonumber\\
    +\delta E_{\mathbb{P}}\langle\xi, \eta\rangle_{\mathbb{R}^n}\;\;\mbox{for}\;\;\xi,\eta\in L^2_{\mathcal{F}^w_T}(\Omega,\mathbb{P};\mathbb{R}^n).
\end{eqnarray}
    It is evident  that $\mathcal{G}_{[s,T]}^\delta$ is self-adjoint and bounded on $L^2_{\mathcal{F}^w_T}(\Omega,\mathbb{P};\mathbb{R}^n)$, satisfying the following  inequality
    for any $\xi\in L^2_{\mathcal{F}_T}(\Omega, \mathbb{P};\mathbb{R}^n)$:
\begin{equation}\label{yu-6-23-3}
    \langle \mathcal{G}_\delta\xi,\xi\rangle_{L^2_{\mathcal{F}^w_T}
    (\Omega,\mathbb{P};\mathbb{R}^n)}
    \geq \delta E_{\mathbb{P}}|\xi|_{\mathbb{R}^n}^2\;\;\mbox{and}\;\;
    \langle \mathcal{G}_\delta\xi,\xi\rangle_{L^2_{\mathcal{F}_T}(\Omega,
    \mathbb{P};\mathbb{R}^n)}
    \geq E_{\mathbb{P}}|y(s;\xi)|^2_{\mathbb{R}^n}.
     \end{equation}
     Consequently, by the Lax-Milgram theorem, we get that $\mathcal{G}_{[s,T]}^\delta$ is invertible,  and
\begin{equation}\label{yu-6-23-3-b}
    \Big\|\big(\mathcal{G}_{[s,T]}^\delta\big)^{-1}
    \Big\|_{\mathcal{L}(L^2_{\mathcal{F}^w_T}
    (\Omega,\mathbb{P};\mathbb{R}^n))}
    \leq \delta^{-1}.
\end{equation}
    Given  an arbitrarily   $x_s\in \mathbb{R}^n$,  we define
\begin{equation}\label{yu-6-23-4}
    f:=\big(\mathcal{G}_{[s,T]}^\delta\big)^{-1}x(T;0,x_s,s) (\in L^2_{\mathcal{F}^w_T}(\Omega,\mathbb{P};\mathbb{R}^n))
\end{equation}
    and
\begin{equation}\label{yu-6-23-5}
    u_f(t):=-c\Big(B^\top y(t;f)+\sum_{i=1}^dD_i^\top Y_i(t;f)\Big),\;\;t\in[s,T].
\end{equation}
    It is evident  that  $u_f\in L^2_{\mathfrak{T}(w,\mathbb{P})}(s,T;\mathbb{R}^m)$,  and
\begin{eqnarray}\label{yu-6-23-6}
    E_{\mathbb{P}}\int_s^T|u_f(t)|^2_{\mathbb{R}^m}dt
    &=&c^2 E_{\mathbb{P}}\int_s^T\Big|B^\top y(t;f)+\sum_{i=1}^dD_i^\top Y_i(t;f)\Big|_{\mathbb{R}^m}^2dt=c\big(\langle \mathcal{G}_{[s,T]}^\delta f,f\rangle_{L^2_{\mathcal{F}^w_T}(\Omega,\mathbb{P};\mathbb{R}^n)} -\delta E_{\mathbb{P}}|f|^2_{\mathbb{R}^n}\big)\nonumber\\
    &\leq& c\big\langle x(T;0,x_s,s), (\mathcal{G}^\delta_{[s,T]})^{-1}x(T;0,x_s,s)\big\rangle_{L^2_{\mathcal{F}^w_T}
    (\Omega,\mathbb{P};\mathbb{R}^n)}.
\end{eqnarray}
     Then, combining \eqref{yu-6-23-3-b}, \eqref{yu-6-23-6} and the definition of $c_0$, we derive
\begin{equation}\label{yu-6-23-7}
    \|u_f\|_{L^2_{\mathfrak{T}(w,\mathbb{P})}(s,T;\mathbb{R}^m)}\leq \sqrt{c\delta^{-1}c_0(T-s)}\|x_s\|_{L^2_{\mathcal{F}^w_s}
    (\Omega,\mathbb{P};\mathbb{R}^n)}.
\end{equation}

    Next, we aim to prove that
\begin{equation}\label{yu-6-23-8}
E_{\mathbb{P}}|x(T;u_f,x_s,s)|^2_{\mathbb{R}^n}\leq \delta\|x_s\|^2_{L^2_{\mathcal{F}^w_s}
    (\Omega,\mathbb{P};\mathbb{R}^n)}.
\end{equation}
    Before doing so, we first establish
\begin{equation}\label{yu-6-23-9}
    E_{\mathbb{P}}|f|_{\mathbb{R}^n}^2\leq \delta^{-1}\|x_s\|^2_{L^2_{\mathcal{F}^w_s}(\Omega,\mathbb{P};\mathbb{R}^n)}.
\end{equation}
    Indeed, by applying It\^{o}'s formula to $\big\langle x(\cdot;0,x_s,s),y(\cdot;
    (\mathcal{G}_{[s,T]}^\delta)^{-1}f)\big\rangle_{\mathbb{R}^n}$  and utilizing the second inequality in \eqref{yu-6-23-3}, we obtain
\begin{eqnarray*}\label{yu-6-23-10}
    E_{\mathbb{P}}|f|_{\mathbb{R}^n}^2&=&E_{\mathbb{P}}\big\langle x(T;0,x_s,s),(\mathcal{G}_{[s,T]}^\delta)^{-1}f\big\rangle_{\mathbb{R}^n}
    =E_{\mathbb{P}}\big\langle x_s,y(s;(\mathcal{G}_{[s,T]}^\delta)^{-1}f)\big\rangle_{\mathbb{R}^n}\nonumber\\
    &\leq& E_{\mathbb{P}}\big(|x_s|_{\mathbb{R}^n}|y(s,(\mathcal{G}_{[s,T]}^\delta)^{-1} f)|_{\mathbb{R}^n}\big)
    \leq  \big(E_{\mathbb{P}}|x_s|^2_{\mathbb{R}^n}\big)^{\frac{1}{2}} \Big(\big\langle f,(\mathcal{G}_{[s,T]}^\delta)^{-1} f\big\rangle_{L^2_{\mathcal{F}_T}(\Omega;\mathbb{R}^n)}\Big)^{\frac{1}{2}}
    \nonumber\\
    &\leq& \big(E_{\mathbb{P}}|x_s|^2_{\mathbb{R}^n}\big)^{\frac{1}{2}}
    \|(\mathcal{G}_{[s,T]}^\delta)^{-1}\|^{\frac{1}{2}}
    _{\mathcal{L}(L^2_{\mathcal{F}^w_T}(\Omega;\mathbb{R}^n))}
(E_{\mathbb{P}}|f|_{\mathbb{R}^n}^2)^{\frac{1}{2}}.
\end{eqnarray*}
       This, combined with \eqref{yu-6-23-3-b}, confirms the validity of \eqref{yu-6-23-9}.
Now, let $\xi\in L^2_{\mathcal{F}^w_T}(\Omega;\mathbb{R}^n)$  be arbitrary. Similar to the proof of \eqref{yu-6-22-11}, we derive
\begin{equation*}\label{yu-6-23-11}
    E_{\mathbb{P}}\langle x(T;u_f,x_s,s),\xi\rangle_{\mathbb{R}^n}-E_{\mathbb{P}}\langle x_s,y(s;\xi)\rangle_{\mathbb{R}^n}
    =E_{\mathbb{P}}\int_s^T\Big\langle u_f(t),B^\top y(t;\xi)+\sum_{i=1}^dD_i^\top Y_i(t;\xi)\Big\rangle_{\mathbb{R}^m}dt.
\end{equation*}
     This, along with \eqref{yu-6-23-2}, \eqref{yu-6-23-4}, and \eqref{yu-6-23-5}, leads to
\begin{eqnarray}\label{yu-6-23-12}
    E_{\mathbb{P}}\langle x(T;u_f,x_s,s),\xi\rangle_{\mathbb{R}^n}&=&E_{\mathbb{P}}\langle x_s,y(s;\xi)\rangle_{\mathbb{R}^n}
    -\langle \mathcal{G}_{[s,T]}^\delta f,\xi\rangle_{L^2_{\mathcal{F}_T}(\Omega;\mathbb{R}^n)}
    +\delta E_{\mathbb{P}}\langle f,\xi\rangle_{\mathbb{R}^n}\nonumber\\
    &=&E_{\mathbb{P}}\langle x_s,y(s;\xi)\rangle_{\mathbb{R}^n}
    -E_{\mathbb{P}}\langle x(T;0,x_s,s),\xi\rangle_{\mathbb{R}^n}
    +\delta E_{\mathbb{P}}\langle f,\xi\rangle_{\mathbb{R}^n}.
\end{eqnarray}
    Note that, using  It\^{o}'s formula, it can be shown that  $E_{\mathbb{P}}\langle x_s,y(s;\xi)\rangle_{\mathbb{R}^n}
    -E\langle x(T;0,x_s,s),\xi\rangle_{\mathbb{R}^n}=0$.  Thus, combining \eqref{yu-6-23-12} and \eqref{yu-6-23-9}, we obtain
\begin{equation*}\label{yu-6-23-13}
    E_{\mathbb{P}}\langle x(T;u_f,x_s,s),\xi\rangle_{\mathbb{R}^n}=\delta E_{\mathbb{P}}\langle f,\xi\rangle_{\mathbb{R}^n}\leq \delta^{\frac{1}{2}}(E_{\mathbb{P}}|\xi|_{\mathbb{R}^n}^2)^{\frac{1}{2}}
    (E_{\mathbb{P}}|x_s|^2_{\mathbb{R}^n})^{\frac{1}{2}}.
\end{equation*}
     Given the arbitrariness of $\xi$,  this implies that
    \eqref{yu-6-23-8} holds.  By \eqref{yu-6-23-7} and \eqref{yu-6-23-8}, we can conclude that the system $[A,B,C,D,\mathfrak{T}(w,\mathbb{P})]$ is  $\delta$-null controllable on $[0,T]$ with cost $\sqrt{c\delta^{-1}c_0(T-s)}$.  In other words, the claim $(i)$ holds. The proof of Theorem \ref{yu-theorem-6-22-1} is now complete.
\end{proof}

\begin{remark} The operator $\mathcal{G}_{[s,T]}^\delta$ defined by \eqref{yu-6-23-2}   can be interpreted as the Gramian operator of a stochastic observation system (cf. \cite{Ma-Wang-Yu} for the deterministic observation system case).
\end{remark}

    We are now ready to present the proof of Theorem \ref{yu-theorem-7-27-2}.
\vskip 5pt
\begin{proof}[\emph{{\bf Proof of Theorem \ref{yu-theorem-7-27-2}}}]
     The proof will be divided into several steps.

    {\bf Step 1:  $(i)\Rightarrow (ii)$.} Suppose that, for any filtered probability space $\mathfrak{T}(w,\mathbb{P})$, the system $[A,B,C,D, \mathfrak{T}(w,\mathbb{P})]$ is feedback stabilizable.  Fix an arbitrary filtered probability space $\mathfrak{T}(w,\mathbb{P})$. By our assumption, there exists  $F\in \mathbb{R}^{m\times n}$ such that, for any  $x_0\in\mathbb{R}^n$, the solution $x_F(\cdot;x_0)$ of the closed-loop system \eqref{yu-7-29-1}
    	satisfies   \eqref{yu-7-29-2}, i.e., there is  $\alpha>0$ and
    $c(\alpha)>0$ such that for each $t\in\mathbb{R}^+$ and any $x_0\in\mathbb{R}^n$ such the following inequality holds:
\begin{equation}\label{yu-8-8-1}
     E_{\mathbb{P}}|x_F(t;x_0)|_{\mathbb{R}^n}^2\leq
   c(\alpha)e^{-\alpha t}|x_0|^2_{\mathbb{R}^n}.
\end{equation}
        Consequently, for every  $\delta\in(0,1)$, there exists $T(\delta)>0$ such that, for any $x_0\in\mathbb{R}^n$ and $t\geq T(\delta)$, the following inequality is satisfied:
\begin{equation}\label{yu-8-8-2}
     E_{\mathbb{P}}|x_F(t;x_0)|_{\mathbb{R}^n}^2\leq\delta
     |x_0|^2_{\mathbb{R}^n}.
\end{equation}
    Let $x_0\in\mathbb{R}^n$ and $\delta\in(0,1)$ be fixed arbitrarily. Denote $T:=T(\delta)$ and $u_F(\cdot):=Fx_F(\cdot;x_0)$. It is evident  that $x_F(\cdot;x_0)=x(\cdot;u_F, x_0)$.
     Furthermore, by \eqref{yu-8-8-1} and \eqref{yu-8-8-2}, we can deduce that
\begin{equation}\label{yu-8-8-4}
	E_{\mathbb{P}}|x(T;u_F,x_0)|_{\mathbb{R}^n}^2\leq\delta
	|x_0|^2_{\mathbb{R}^n}
\end{equation}
    and
\begin{eqnarray}\label{yu-8-8-3}
 \|u_F\|_{L^2_{\mathfrak{T}(w,\mathbb{P})}(0,T;\mathbb{R}^m)}=
 \Big(E_{\mathbb{P}}\int_0^T  |Fx(t;u_F,x_0)|^2_{\mathbb{R}^m}dt\Big)^{\frac{1}{2}}\leq c(\delta,T)|x_0|_{\mathbb{R}^n},
\end{eqnarray}
 where $c(\delta,T):=\alpha^{-1}c(\alpha)(1-e^{-\alpha T})|F|_{\mathbb{R}^{m\times n}}$.

Given the arbitrariness of $x_0$,  \eqref{yu-8-8-4} and \eqref{yu-8-8-3} imply that $[A,B,C,D,\mathfrak{T}(w,\mathbb{P})]$ is $\delta$-null controllable with cost in $[0,T]$.  This, combined with the arbitrariness of $\mathfrak{T}(w,\mathbb{P})$ and $\delta$, leads to the conclusion that, for any $\mathfrak{T}(w,\mathbb{P})$,   $[A,B,C,D,\mathfrak{T}(w,\mathbb{P})]$ is approximately null controllable with cost.  Hence, $(ii)$ holds.

    {\bf Step 2:  $(ii)\Rightarrow (iii)$.}  This implication is straightforward.

    {\bf Step 3:   $(iii)\Rightarrow (iv)$.}  This is evident by Theorem \ref{yu-theorem-6-22-1}.

    {\bf Step 4:  $(iv)\Rightarrow (i)$.} Suppose that the system $[A^\top,B^\top,C^\top,D^\top,\mathfrak{T}(w,\mathbb{P})]$ is weakly observable for a filtered probability space $\mathfrak{T}(w,\mathbb{P})$. This means that there exist  $T\in \mathbb{R}^+$, $\delta\in[0,1)$ and $c(\delta,T)\in \mathbb{R}^+$ such that the weak observability inequality \eqref{yu-6-22-4-001} holds on $\mathfrak{T}(w,\mathbb{P})$.  Therefore, according to Theorem \ref{yu-theorem-7-24-1}, the weak observability inequality \eqref{yu-6-22-4-001} holds on any filtered probability space $\mathfrak{T}(w,\mathbb{P})$ with the same $T$, $\delta$ and $c(\delta,T)$.

    Let $\mathfrak{T}(w,\mathbb{P})$ be fixed arbitrarily.
    For every  $k\in\mathbb{N}$, we define
    $\widehat{w}_{k}(\cdot)=w(kT+\cdot)-w(kT)$ in $[0,+\infty)$ and $\mathfrak{T}_k:=\mathfrak{T}(\widehat{w}_{k}, \mathbb{P}_k)$ with $\mathbb{P}_k:=\mathbb{P}(\cdot|\mathcal{F}^w_{kT})$ and $\{\mathcal{F}^{\widehat{w}_k}_t\}_{t\in[0,T]}$ to be the natural filtration on $\mathfrak{T}_k$. (Here, we note that $\widehat{w}_k$ is a standard Brownain motion on $(\Omega,\mathcal{F},\mathbb{P}_k)$.)  It is clear  that $\widehat{w}_0=w$ and $\mathfrak{T}_0=\mathfrak{T}(w,\mathbb{P})$.  By Theorem \ref{yu-theorem-7-24-1} and our assumption, we can assert that, for every  $k\in\mathbb{N}$,
\begin{equation*}\label{yu-8-8-10}
	E_{\mathbb{P}_k}|y^{\widehat{w}_{k}}(0,y_1)|^2_{\mathbb{R}^n}
	\leq c(\delta,T)E_{\mathbb{P}_k}\int_0^T\Big|B^\top y^{\widehat{w}_{k}}(t,y_1)+\sum_{i=1}^dD_i^\top Y^{\widehat{w}_{k}}_i(t;y_1)\Big|^2_{\mathbb{R}^m}dt
	+\delta E_{\mathbb{P}_k}|y_1|_{\mathbb{R}^n}^2
\end{equation*}
    for any $y_1\in L^2_{\mathcal{F}_T^{\widehat{w}_k}}(\Omega,\mathbb{P}_k;\mathbb{R}^n)$, where
    $(y^{\widehat{w}_{k}}(\cdot,y_1),Y^{\widehat{w}_k}(\cdot;y_1):=(Y_i^{\widehat{w}_k}
    (\cdot;y_1))_{i=1}^d)$ is the solution of the dual system in \eqref{yu-7-26-10} on $\mathfrak{T}_k$.
     Combining this with Proposition \ref{yu-proposition-6-27-1}, we obtain that, for any $k\in\mathbb{N}$,
\begin{eqnarray}\label{yu-bb-8-8-1-2}
    E_{\mathbb{P}_k}|y^{k+1}(kT;y_1)|_{\mathbb{R}^n}^2
    &\leq& c(\delta,T)E_{\mathbb{P}_k}\int_0^T\Big|B^\top y^{k+1}(kT+t,y_1)+\sum_{i=1}^dD_i^\top Y^{k+1}_i(kT+t;y_1)\Big|^2_{\mathbb{R}^m}dt+\delta E_{\mathbb{P}_k}|y_1|_{\mathbb{R}^n}^2\nonumber\\
    &=& c(\delta,T)E_{\mathbb{P}_k}\int_{kT}^{(k+1)T}\Big|B^\top y^{k+1}(t,y_1)+\sum_{i=1}^dD_i^\top Y^{k+1}_i(t;y_1)\Big|^2_{\mathbb{R}^m}dt\nonumber\\
    &\;&+\delta E_{\mathbb{P}_k}|y_1|_{\mathbb{R}^n}^2
   \;\;\mathbb{P}-a.s.,\;\;\mbox{on}\;\;(\Omega,\mathcal{F},\mathbb{P})
    \;\;\mbox{for any}\;\;y_1\in L^2_{\mathcal{F}^w_{(k+1)T}}(\Omega,\mathbb{P};\mathbb{R}^n),
\end{eqnarray}
    where $(y^k(\cdot;y_1),Y^k(\cdot;y_1):=(Y_i^k
    (\cdot;y_1))_{i=1}^d))$ with $y_1\in L^2_{\mathcal{F}^w_{kT}}
    (\Omega,\mathbb{P};\mathbb{R}^n)$ is the solution of the equation \eqref{yu-6-22-1}. Taking
    the expectation $E_{\mathbb{P}}$ on  both sides of
    \eqref{yu-bb-8-8-1-2}, we obtain
\begin{eqnarray*}\label{yu-bb-8-8-1}
     E_{\mathbb{P}}|y^{k+1}(kT;y_1)|_{\mathbb{R}^n}^2
     &\leq&  c(\delta,T)E_{\mathbb{P}}\int_{kT}^{(k+1)T}\Big|B^\top y^{k+1}(t,y_1)+\sum_{i=1}^dD_i^\top Y^{k+1}_i(t;y_1)\Big|^2_{\mathbb{R}^m}dt\nonumber\\
    &\;&+\delta E_{\mathbb{P}}|y_1|_{\mathbb{R}^n}^2
    \;\;\mbox{for any}\;\;y_1\in L^2_{\mathcal{F}^w_{(k+1)T}}
    (\Omega,\mathbb{P};\mathbb{R}^n).
\end{eqnarray*}
     From Theorem \ref{yu-theorem-6-22-1}, we have the following conclusion
     \textbf{(CL)}: \emph{for any  $k\in\mathbb{N}$ and   $x_k\in
    L^2_{\mathcal{F}_{kT}^w}(\Omega,\mathbb{P};\mathbb{R}^n)$, there exists $u_k\in L^2_{\mathfrak{T}(w,\mathbb{P})}(kT,(k+1)T;\mathbb{R}^m)$ such that
    \begin{equation*}
		\|u_k\|^2_{L^2_{\mathfrak{T}(w,\mathbb{P})}(kT,(k+1)T;\mathbb{R}^m)}\leq c_{\delta,T}E_{\mathbb{P}}|x_k|^2_{\mathbb{R}^n}\;\;\mbox{and}\;\;
    			E_{\mathbb{P}}|x((k+1)T;u_k,x_k, kT)|^2_{\mathbb{R}^n}
\leq \delta E_{\mathbb{P}}|x_k|^2_{\mathbb{R}^n},
    	\end{equation*}
     where $c_{\delta,T}:=c(\delta,T)\delta^{-1}c_0(T)$. (Here,  $c_0(\cdot)$  is given such that \eqref{yu-b-8-8-1} holds (see Remark \ref{yu-remark-4-2}) and $x(\cdot;u,x_k,kT)$ is the solution of the equation \eqref{yu-7-3-1}
    with   $s$, $T$ and $x_s$ replaced  by $kT$, $(k+1)T$ and $x_k$ respectively).}

    Let $x_0\in\mathbb{R}^n$ be fixed arbitrarily.
    In the interval $[0,T]$, by \textbf{(CL)}, there exists  $u_0\in L^2_{\mathfrak{T}(w,\mathbb{P})}(0,T;\mathbb{R}^m))$ such that
\begin{equation}\label{yu-8-8-11}
    			\|u_0\|^2_{L^2_{\mathfrak{T}_0}(0,T;\mathbb{R}^m)}\leq c_{\delta,T}|x_0|^2_{\mathbb{R}^n}\;\;\mbox{and}\;\;
    		E_{\mathbb{P}}|x(T;u_0,x_0)|^2_{\mathbb{R}^n}
\leq \delta|x_0|^2_{\mathbb{R}^n}.
    	\end{equation}
    Let $x_1:=x(T;u_0,x_0)$. In the interval $[T,2T]$, by \textbf{(CL)} again, there exists
    $u_1\in L^2_{\mathfrak{T}(w,\mathbb{P})}(T,2T;\mathbb{R}^m)$ such that
\begin{equation*}
    		\begin{cases}
		\|u_1\|^2_{L^2_{\mathfrak{T}(w,\mathbb{P})}(T,2T;\mathbb{R}^m)}\leq c_{\delta,T}E_{\mathbb{P}}|x_1|^2_{\mathbb{R}^n}\leq c_{\delta,T}\delta |x_0|^2_{\mathbb{R}^n},\\
    			E_{\mathbb{P}}|x(2T;u_1,x_1, T)|^2_{\mathbb{R}^n}
\leq \delta E_{\mathbb{P}}|x_1|^2_{\mathbb{R}^n}\leq \delta^2 |x_0|^2_{\mathbb{R}^n}.
    		\end{cases}
    	\end{equation*}
      By applying an induction argument, we can construct two sequences $\{x_k\}_{k\in\mathbb{N}}$ and $\{u_k\}_{k\in\mathbb{N}}$ such that
    \begin{equation}\label{yu-8-8-19}
    		\begin{cases}
    x_{k+1}:=x((k+1)T; u_k, x_k, kT),\\
		\|u_{k}\|^2_{L^2_{\mathfrak{T}(w,\mathbb{P})}(kT,(k+1)T;\mathbb{R}^m)}\leq c_{\delta,T}\delta^k |x_0|^2_{\mathbb{R}^n},\\
    			E_{\mathbb{P}}|x((k+1)T;u_k,x_k, kT)|^2_{\mathbb{R}^n}
\leq \delta^{k+1} |x_0|^2_{\mathbb{R}^n},
    		\end{cases}
        \mbox{ for any  }k\in\mathbb{N},
    	\end{equation}
    where $u_0$ is given so that \eqref{yu-8-8-11} holds.  We define
$u^*(t)=u_k(t)$ if $t\in[kT,(k+1)T)$.
      From \eqref{yu-8-8-19}, it is clear that $u^*\in L^2_{\mathfrak{T}(w,\mathbb{P})}(\mathbb{R}^+;\mathbb{R}^m)$ with
\begin{equation}\label{yu-8-8-31}
    \|u^*\|^2_{L^2_{\mathfrak{T}(w,\mathbb{P})}(\mathbb{R}^+;\mathbb{R}^m)}
    \leq c_{\delta,T}|x_0|_{\mathbb{R}^n}^2\sum_{k=0}^{+\infty}\delta^k=
    c_{\delta,T}(1-\delta)^{-1}|x_0|_{\mathbb{R}^n}^2.
\end{equation}
     Furthermore, it is straightforward to verify that
$x(t;u^*,x_0)=x(t;u_k,x_k,kT)$ for any $t\in[kT,(k+1)T)$.
     Combining this with the variation of constants formula for SDEs, we obtain, for every  $k\in\mathbb{N}$,
    \begin{equation}\label{zhang-8-8-33}
    x(t;u^*,x_0)=e^{A(t-kT)}x_k+\int_{kT}^t e^{A(t-s)}Bu_k(s)ds+\int_{kT}^t e^{A(t-s)}(Cx(t;u^*,x_0)+Du_k(s))dw(s)
    \end{equation}
    for any $t\in[kT,(k+1)T)$. Therefore, by applying H\"{o}lder's inequality, Gronwall's inequality, and \eqref{zhang-8-8-33}, we can deduce that there exists a constant  $c_T:=c(T)>0$ such that
\begin{equation*}\label{yu-8-8-34}
    E_{\mathbb{P}}\int_{kT}^{(k+1)T}|x(t;u^*,x_0)|^2_{\mathbb{R}^n}dt
    \leq c_T(E_{\mathbb{P}}|x_k|^2_{\mathbb{R}^n}
    +\|u_k\|^2_{L^2_{\mathfrak{T}(w,\mathbb{P})}(kT,(k+1)T;\mathbb{R}^m)})
    \;\;\mbox{for each}\;\;k\in\mathbb{N}.
\end{equation*}
    This, combined with \eqref{yu-8-8-19}, implies that
\begin{equation}\label{yu-8-8-35}
    E_{\mathbb{P}}\int_{\mathbb{R}^+}|x(t;u^*,x_0)|^2_{\mathbb{R}^n}dt
    \leq  c_T(1+c_{\delta,T})|x_0|^2_{\mathbb{R}^n}\sum_{k=0}^{+\infty}\delta^k
    =c_T(1-\delta)^{-1}(1+c_{\delta,T})|x_0|^2_{\mathbb{R}^n}.
\end{equation}
     Therefore, due to the arbitrariness of  $x_0$, \eqref{yu-8-8-35} and \eqref{yu-8-8-31} allow us to conclude that the problem
       \textbf{(LQ)$_{\mathfrak{T}(w,\mathbb{P})}$} defined in Section \ref{yu-sec-2-2}  is solvable. By Lemma \ref{yu-lemma-6-14-1}, there exists $F^*\in\mathbb{R}^{m\times n}$ such that the solution of the closed-loop system \eqref{yu-7-29-1} with $F=F^*$ belongs to $L^2_{\mathfrak{T}(w,\mathbb{P})}(\mathbb{R}^+;\mathbb{R}^n)$.  Together with \cite[Proposition 3.6]{Huang-Li-Yong-2015} (or \cite[Lemma 2.3]{Sun-Yong-Zhang-2021}), this implies that the closed-loop system \eqref{yu-7-29-1} with  $F=F^*$ is asymptotically stable. Hence, the sytem $[A,B,C,D,\mathfrak{T}(w,\mathbb{P})]$ is feedback stabilizable. Since  $\mathfrak{T}(w,\mathbb{P})$  is arbitrary, the claim $(i)$ is true. The proof is complete.
\end{proof}

\section{Appendix: Proof of Lemma  \ref{yu-proposition-7-11-1}}\label{yu-sec-06}

     In this section, we provide the proof of Lemma \ref{yu-proposition-7-11-1} within the context of Section \ref{yu-sect-8-22}. Before proceeding, we recall some fundamental notations and concepts in Malliavin calculus.

    Let $s>0$, $p,q\in\mathbb{N}^+$ and $C^{k}_b(\mathbb{R}^{p};\mathbb{R}^q)$ ($k\in\mathbb{N}^+$) be the set of those functions in $C^{k}(\mathbb{R}^{p};\mathbb{R}^q)$ whose partial derivatives of order less than or equal to $k$ are bounded. We define
$$
    \mathcal{S}:=\{\xi=f(w(g_1),\cdots,w(g_p)): p\in\mathbb{N}^+,\;
    f\in C_b^{\infty}(\mathbb{R}^p;\mathbb{R}^q)
    (=\cap_{k\in\mathbb{N}^+}C_b^k(\mathbb{R}^p;\mathbb{R}^q)),\;(g_i)_{i=1}^p\in (L^2(s,T;\mathbb{R}^d))^p\},
$$
    where $w(g):=\int_s^Tg^\top(\sigma) dw(\sigma)$ for $g\in L^2(s,T;\mathbb{R}^d)$
    (this represents  the Wiener integral of $g^\top$ with respect to Brownian motion $w$). For  any
$\xi=(\xi_1,\xi_2,\cdots,\xi_q)^\top\in \mathcal{S}$,
    with $\xi_j=f_j(w(g_1),w(g_2),\cdots,w(g_p))$ for $1\leq j\leq q$ and some $p\in\mathbb{N}$, the  derivative $\mathbb{D}\xi:=(\mathbb{D}_t\xi)_{t\in[s,T]}(\in L^2(\Omega\times(s,T), \mathbb{P}\times \lambda;\mathbb{R}^{q\times d}))$ of $\xi$  is defined by $\mathbb{D}_t\xi=((\mathbb{D}_t\xi_1)^\top,(\mathbb{D}_t\xi_2)^\top,\cdots,
    (\mathbb{D}_t\xi_q)^\top)^\top$ with
\begin{equation}\label{yu-7-11-5}
    \mathbb{D}_t\xi_k:=\sum_{j=1}^p\partial_{\zeta_j}f_{k}(w(g_1),w(g_2),
    \cdots,w(g_p))g_j(t),\;\;t\in[s,T].
\end{equation}
       Here and in subsequent discussions, $\partial_{\zeta_j}f$  denotes the partial derivative of $f$ with respect to the $j$-th variable.
     It is evident that $\mathbb{D}_t\xi$ is a  valued function in $\mathbb{R}^{q\times d}$.
     Furthermore, we denote  the $i$-th column of $\mathbb{D}_t\xi$ as  $\mathbb{D}^i_t\xi$ (for $1\leq i\leq d$).  It is well-established that $\mathbb{D}$, as an unbounded operator from
       from $L^2_{\mathcal{F}^{w,s}_T}(\Omega,\mathbb{P};\mathbb{R}^q)$ to $L^2(\Omega\times(s,T), \mathbb{P}\times \lambda;\mathbb{R}^{q\times d}))$ is closable.
       Consequently, it can be extended to a closed operator from  $L^2_{\mathcal{F}^{w,s}_T}(\Omega,\mathbb{P};\mathbb{R}^q)$ to $L^2(\Omega\times(s,T), \mathbb{P}\times \lambda;\mathbb{R}^{q\times d})$.  This extension retains the same notation.
    On the set $\mathcal{S}$, we define a norm by
$\|\xi\|_{1,2}:=\left(E_{\mathbb{P}}|\xi|_{\mathbb{R}^q}^2+
    \sum_{i=1}^q E_{\mathbb{P}}\int_s^T|\mathbb{D}_t\xi_i|_{\mathbb{R}^d}^2dt
    \right)^{\frac{1}{2}}$
    ($\xi=(\xi_1,\xi_2,\cdots,\xi_q)^\top\in\mathcal{S}$).
    We define  $\mathcal{D}_{\mathfrak{T}^s(w,\mathbb{P}),q}^{1,2}(s,T)$ as the completion of $\mathcal{S}$
     with the norm $\|\cdot\|_{1,2}$ (i.e., $\mathcal{D}_{\mathfrak{T}^s(w,\mathbb{P}),q}^{1,2}(s,T)=
    \overline{\mathcal{S}}^{\|\cdot\|_{1,2}}$).
     Consequently, the domain of $\mathbb{D}$ in $L^2_{\mathcal{F}^{w,s}_T}(\Omega,\mathbb{P};\mathbb{R}^q)$ is $\mathcal{D}_{\mathfrak{T}^s(w,\mathbb{P}),q}^{1,2}(s,T)$.

    In this section, we define
\begin{eqnarray*}
    L^p_{\mathfrak{T}^s(w,\mathbb{P})}(\Omega; L^2(s,T;\mathbb{R}^q)):=\Big\{
    \xi:[s,T]\times \Omega\to\mathbb{R}^q:\xi
    \;\mbox{is}\;\{\mathcal{F}^{w,s}_t\}_{t\geq 0}\mbox{-progresively measurable}
     \\
     \mbox{and}
    \;E_{\mathbb{P}}\Big(\int_s^T|\xi(t)|_{\mathbb{R}^q}^2dt\Big)^{\frac{p}{2}}
    <+\infty
    \Big\},
\end{eqnarray*}
    and
\begin{eqnarray*}
    L^p_{\mathfrak{T}^s(w,\mathbb{P})}(\Omega;C([s,T];\mathbb{R}^q)):
    =\Big\{\xi:[s,T]\times\Omega
    \to\mathbb{R}^q:\xi
    \;\mbox{is}\;\{\mathcal{F}^{w,s}_t\}_{t\geq 0}\mbox{-adapted, contunous}\\
    \mbox{from}\;[s,T]\;\mbox{to}\; \mathbb{R}^q,\;\mathbb{P}\mbox{-a.s.}
    \;\mbox{and}\;E_{\mathbb{P}}\Big(
    \sup_{t\in[s,T]}|\xi(t)|^p_{\mathbb{R}^q}\Big)<+\infty\Big\}.
\end{eqnarray*}
    In this section, we let $(x^{s,\zeta},y^{s,\zeta},Y^{s,\zeta}=(Y_i^{s,\zeta})_{i=1}^d)$ be the solution of the FBSDE \eqref{yu-7-9-3}.
\begin{lemma}\label{yu-lemma-7-23-1}
      For every  $p\geq 1$, $(y^{s,\zeta},(Y_i^{s,\zeta})_{i=1}^d)
     \in L^p_{\mathfrak{T}^s(w,\mathbb{P})}(\Omega;C([s,T];\mathbb{R}^n))
     \times (L^p_{\mathfrak{T}^s(w,\mathbb{P})}(\Omega; L^2(s,T;\mathbb{R}^n)))^d$,  and  there exists  a constant $c(p,T)>0$ such that, for any $(s,\zeta)\in[0,T]\times \mathbb{R}^{nl}$, the following inequality holds
\begin{equation*}\label{yu-7-11-10}
    E_{\mathbb{P}}\big(\sup_{s\leq t\leq T}|y^{s,\zeta}(t)|_{\mathbb{R}^n}^p\big)
    +\sum_{i=1}^d E_{\mathbb{P}}\Big(\int_s^T|Y_i^{s,\zeta}(t)|^2_{\mathbb{R}^n}dt
    \Big)^{\frac{p}{2}}\leq c(p,T)E_{\mathbb{P}} |F_{l,N}(x^{s,\zeta}(T))|_{\mathbb{R}^n}^p.
\end{equation*}
\end{lemma}
    The proof of this lemma follows the standard three-step strategy outlined in \cite{Pardoux-Peng-1990} (see also Lemma 2.1 in \cite{Pardoux-Peng-1992}). However, due to the linearity of the forward-backward stochastic differential equation (FBSDE) \eqref{yu-7-9-3}, only two steps are necessary. Therefore, we omit the detailed proof in this paper.

\begin{lemma}\label{yu-lemma-7-23-2}
     Given $(s,\zeta)\in[0,T]\times \mathbb{R}^{nl}$, the following statements hold:
    \begin{enumerate}
      \item [$(a)$] $(y^{s,\zeta}, Y^{s,\zeta})\in L^2(s,T;\mathcal{D}_{\mathfrak{T}^s(w,\mathbb{P}),n}^{1,2}(s,T))
          \times (L^2(s,T;\mathcal{D}_{\mathfrak{T}^s(w,\mathbb{P}),n}^{1,2}(s,T)))^d$ and, for any $\theta\in[s,T]$, $j\in\{1,2,\ldots,d\}$,
      \begin{equation}\label{yu-7-16-6}
    (\mathbb{D}_\theta^jy^{s,\zeta}, \mathbb{D}_\theta^jY^{s,\zeta}:=(\mathbb{D}_\theta^jY^{s,\zeta}_i)_{i=1}^d)
    \in L^2(\Omega\times(s,T), \mathbb{P}\times \lambda;\mathbb{R}^{n})\times L^2(\Omega\times(s,T), \mathbb{P}\times\lambda;\mathbb{R}^{n\times d})
\end{equation}
 and
\begin{eqnarray}\label{yu-7-16-7}
     \mathbb{D}_\theta^jy^{s,\zeta}(t)
    &=&\nabla F_{l,N}(x^{s,\zeta}(T))
    h^\top (\theta)f_j+\int_t^T\Big(A^\top \mathbb{D}_\theta^jy^{s,\zeta}(\sigma)+\sum_{i=1}^dC_i^\top \mathbb{D}_\theta^jY^{s,\zeta}_i(\sigma)\Big)d\sigma\nonumber\\
    &\;&-\sum_{i=1}^d\int_t^T\mathbb{D}_\theta^jY^{s,\zeta}_i(\sigma)dw^i(\sigma)
    \;\;\mbox{for any}\;\;t\in
    [\theta,T].
    \end{eqnarray}
       \item[$(b)$] $y^{s,\zeta}(\cdot)$ and $\mathbb{D}^{j}_\theta y^{s,\zeta}(\cdot)$ ($j\in\{1,2,\ldots,d\}$) have continuous versions over $[s,T]$ and $[\theta,T]$ respectively, where $\theta\in[s,T]$.
       \item[$(c)$] $\mathbb{D}^{j}_\theta y^{s,\zeta}(\theta)=Y_j^{s,\zeta}(\theta)$ $\mathbb{P}$-a.s. for any $\theta\in[s,T]$ and $j\in\{1,2,\ldots,d\}$.
    \end{enumerate}
\end{lemma}
\begin{proof}[\emph{\textbf{Proof}}]
     We divide the proof into three steps.

 {\bf Step 1:  The proof of $(a)$.}  It is  evident that  $x^{s,\zeta}\in \mathcal{D}_{\mathfrak{T}^s(w,\mathbb{P}),nl}^{1,2}(s,T)$ and satisfies
\begin{equation}\label{yu-7-15-1}
    \mathbb{D}^j_tx^{s,\zeta}(T)=h^\top (t)f_j,\;\;\mbox{for any}\;\;t\in[s,T]
    \;\;\mbox{and each}\;\;j\in\{1,2,\ldots,d\}.
\end{equation}
    Therefore,  $\mathbb{D}x^{s,\zeta}(T)\in C([s,T];\mathbb{R}^{(nl)\times d})$ and
\begin{equation}\label{yu-7-15-2}
    \mathbb{D}F_{l,N}(x^{s,\zeta}(T))=\nabla F_{l,N}(x^{s,\zeta}(T))
    \mathbb{D}
    x^{s,\zeta}(T)\in (C([s,T];L^2_{\mathcal{F}^{w,s}_T}(\Omega,\mathbb{P}; \mathbb{R}^n)))^d,
\end{equation}
     i.e., $F_{l,N}(x^{s,\zeta}(T))\in
     \mathcal{D}_{\mathfrak{T}^s(w,\mathbb{P}),n}^{1,2}(s,T)^d$. Moreover,
     by Lemma \ref{yu-lemma-7-23-1}, it is well known from \eqref{yu-7-11-5} that, for
     any  $\theta\in(s,T]$,
$(\mathbb{D}_\theta^jy^{s,\zeta}(t), \mathbb{D}_\theta^jY^{s,\zeta}(t))=0$
    for any $t\in[s,\theta)$.
      Now, let
$$
\begin{cases}
    y^0\in L^2_{\mathfrak{T}^s(w,\mathbb{P})}(\Omega;C([s,T];\mathbb{R}^n))\cap L^2(s,T;\mathcal{D}^{1,2}_{\mathfrak{T}^s(w,\mathbb{P}),n}(s,T)),\\
    Y^0:=(Y^0_i)_{i=1}^d\in (L^2_{\mathfrak{T}^s(w,\mathbb{P})}(s,T;\mathbb{R}^n))^d\cap (L^2(s,T;\mathcal{D}^{1,2}_{\mathfrak{T}^s(w,\mathbb{P}),n}(s,T))^d
\end{cases}
$$
be arbitrarily.
    By \eqref{yu-7-15-2}, we have
\begin{equation}\label{yu-7-15-3}
    F_{l,N}(x^{s,\zeta}(T))+\int_{s}^T\Big(A^\top y^0(\sigma)+\sum_{i=1}^dC_i^\top Y^0_i(\sigma)\Big)d\sigma\in \mathcal{D}^{1,2}_{\mathfrak{T}^s(w,\mathbb{P}),n}(s,T).
\end{equation}
     By adopting the three-step strategy outlined in \cite{Pardoux-Peng-1990}, we can construct a Cauchy sequence
    $\{(y^k,Y^k:=(Y^k_i)_{i=1}^d)\}_{k\in\mathbb{N}^+}\in L^2_{\mathfrak{T}^s(w,\mathbb{P})}(\Omega;C([s,T];\mathbb{R}^n))\times (L^2_{\mathfrak{T}^s(w,\mathbb{P})}(s,T;\mathbb{R}^n))^d$ such that
    $(1)$ it holds that
    \begin{eqnarray}\label{yu-7-15-10}
    (y^k,Y^k)
    \to (y^{s,\zeta},Y^{s,\zeta})\;\;\mbox{as}\;\;k\to+\infty
\end{eqnarray}
    (in $L^2_{\mathfrak{T}^s(w,\mathbb{P})}(\Omega;C([s,T];\mathbb{R}^n))\times (L^2_{\mathfrak{T}^s(w,\mathbb{P})}(s,T;\mathbb{R}^n))^d$); $(2)$  For any  $t\in[s,T]$ and $k\in\mathbb{N}^+$,
\begin{eqnarray}\label{yu-7-13-5}
     \sum_{i=1}^d\int_t^TY^{k+1}_i(\sigma)dw^i(\sigma)
    =F_{l,N}(x^{s,\zeta}(T))-y^{k+1}(t)
    +\int_t^T\Big(A^\top y^k(\sigma)+\sum_{i=1}^dC_i^\top Y^k_i(\sigma)\Big)d\sigma;
\end{eqnarray}
  $(3)$ For each $k\in\mathbb{N}^+$,
\begin{eqnarray}\label{yu-7-15-7}
    \sum_{i=1}^d\int_s^TY^{k+1}_i(\sigma)dw^i(\sigma)
    &=&F_{l,N}(x^{s,\zeta}(T))+\int_{s}^T\Big(A^\top y^k(\sigma)+\sum_{i=1}^dC_i^\top Y^k_i(\sigma)\Big)d\sigma\nonumber\\
    &\;&- E_{\mathbb{P}}\Big[F_{l,N}(x^{s,\zeta}(T))+\int_{s}^T\Big(A^\top y^k(\sigma)+\sum_{i=1}^dC_i^\top Y^k_i(\sigma)\Big)d\sigma\Big]
\end{eqnarray}
    holds.  By \eqref{yu-7-15-3}, \eqref{yu-7-15-7}, and applying
    \cite[Lemma 2.3]{Pardoux-Peng-1992}, we can deduce that
$Y^1_j\in L^2(s,T;\mathcal{D}^{1,2}_{\mathfrak{T}^s(w,\mathbb{P}),n}(s,T))$
    for each $j\in\{1,2,\ldots,d\}$.
Consequently, from   \eqref{yu-7-13-5}, we can infer that
$y^1(t)
    \in \mathcal{D}^{1,2}_{\mathfrak{T}^s(w,\mathbb{P}),n}(s,T)$ for any $t\in[s,T]$.
     By applying an inductive argument, we can demonstrate that, for any  $k\in\mathbb{N}$,
$Y^k_i\in L^2(s,T;\mathcal{D}^{1,2}_{\mathfrak{T}^s(w,\mathbb{P}),n}(s,T))$
     for each $i\in\{1,2,\ldots,d\}$ and
     $y^k(t)
    \in \mathcal{D}^{1,2}_{\mathfrak{T}^s(w,\mathbb{P}),n}(s,T)$ for any $t\in[s,T]$.
\
  Moreover, utilizing \eqref{yu-7-13-5} and \eqref{yu-7-15-2}, for every $\theta\in [s,T)$ and $j\in\{1,2,\ldots,d\}$, we have
\begin{eqnarray*}\label{yu-7-16-1}
    \mathbb{D}_\theta^jy^{k+1}(t)
    &=&\nabla F_{l,N}(x^{s,\zeta}(T))
    \mathbb{D}^j_{\theta}
    x^{s,\zeta}(T)+\int_t^T\Big(A^\top \mathbb{D}_\theta^jy^k(\sigma)+\sum_{i=1}^dC_i^\top \mathbb{D}_\theta^jY^k_i(\sigma)\Big)d\sigma\nonumber\\
    &\;&-\sum_{i=1}^d\int_t^T\mathbb{D}_\theta^jY^k_i(\sigma)dw^i(\sigma)
    \;\;\mbox{for any}\;\;t\in[\theta,T].
\end{eqnarray*}
            By adhering to the first step of the three-step strategy presented in \cite{Pardoux-Peng-1990}, we can deduce that, for any    $k\in\mathbb{N}^+$,
\begin{eqnarray}\label{yu-7-13-6}
    &\;&E_{\mathbb{P}}|\mathbb{D}_\theta^jy^{k+1}(t)
    -\mathbb{D}_\theta^jy^k(t)|^2_{\mathbb{R}^n}
    +\sum_{i=1}^d E_{\mathbb{P}}
    \int_t^T|\mathbb{D}_\theta^jY_i^{k+1}(\sigma)
    -\mathbb{D}_\theta^jY_i^k(\sigma)|^2_{\mathbb{R}^n}d\sigma\nonumber\\
    &\leq&(T-t)c
    \biggl(E_{\mathbb{P}}\int_t^T|\mathbb{D}_\theta^jy^{k}(\sigma)
    -\mathbb{D}_\theta^jy^{k-1}(\sigma)|^2_{\mathbb{R}^n}d\sigma+\sum_{i=1}^d E_{\mathbb{P}}
    \int_t^T|\mathbb{D}_\theta^jY_i^{k}(\sigma)
    -\mathbb{D}_\theta^jY_i^{k-1}(\sigma)|^2_{\mathbb{R}^n}d\sigma\biggl).
\end{eqnarray}
    It follows that the sequence$\{(\mathbb{D}_\theta^jy^{k}, \mathbb{D}_\theta^jY^k:=(\mathbb{D}_\theta^jY^k_i)_{i=1}^d)\}_{k\in\mathbb{N}^+}$
      forms a Cauchy sequence in  $L^2(\Omega\times((T-r_0)\vee\theta,T), \mathbb{P}\times \lambda;\mathbb{R}^{n})\times L^2(\Omega\times((T-r_0)\vee\theta,T), \mathbb{P}\times \lambda;\mathbb{R}^{n\times d})$,
    where $r_0 \in(0,(c^{-1}+\frac{1}{4})^{\frac{1}{2}}-\frac{1}{2})$ is fixed.
    Moreover, by \eqref{yu-7-13-6} and \eqref{yu-7-15-10}, we also have
    $(y^k,Y^k)\in L^2((T-r_0),T;\mathcal{D}_{\mathfrak{T}^s(w,\mathbb{P}),n}^{1,2}(s,T)
    \times (L^2(T-r_0,T;\mathcal{D}_{\mathfrak{T}^s(w,\mathbb{P}),n}^{1,2}(s,T))^d$.
  Therefore, by the closedness of $\mathbb{D}$, \eqref{yu-7-13-6} and \eqref{yu-7-15-10},  we can conclude that $(y^{s,\zeta}, Y^{s,\zeta})\in L^2(T-r_0,T;\mathcal{D}_{\mathfrak{T}^s(w,\mathbb{P}),n}^{1,2}(s,T))
          \times (L^2(T-r_0,T;\mathcal{D}_{\mathfrak{T}^s(w,\mathbb{P}),n}^{1,2}(s,T)))^d$,
\begin{eqnarray}\label{yu-7-16-3}
    (\mathbb{D}_\theta^jy^{s,\zeta}, \mathbb{D}_\theta^jY^{s,\zeta})
    \in L^2(\Omega\times((T-r_0)\vee\theta,T), \mathbb{P}\times \lambda;\mathbb{R}^{n})\nonumber\\
    \times L^2(\Omega\times((T-r_0)\vee\theta,T), \mathbb{P}\times \lambda;\mathbb{R}^{n\times d})
\end{eqnarray}
and
\begin{eqnarray}\label{yu-7-16-2}
    \mathbb{D}_\theta^jy^{s,\zeta}(t)
    &=&\nabla F_{l,N}(x^{s,\zeta}(T))
    \mathbb{D}^j_{\theta}
    x^{s,\zeta}(T)+\int_t^T\Big(A^\top \mathbb{D}_\theta^jy^{s,\zeta}(\sigma)+\sum_{i=1}^dC_i^\top \mathbb{D}_\theta^jY_i^{s,\zeta}(\sigma)\Big)d\sigma\nonumber\\
    &\;&-\sum_{i=1}^d\int_t^T\mathbb{D}_\theta^jY^{s,\zeta}_i(\sigma)dw^i(\sigma)
    \;\;\mbox{for any}\;\;t\in
    \begin{cases}[T-r_0,T],&\mbox{if}\;\;l(\theta)>0,\\
    [\theta,T],&\mbox{if}\;\;l(\theta)=0,
    \end{cases}
\end{eqnarray}
    where  $l(\theta):=
    \left[\frac{T-\theta}{r_0}\right]$. If $l(\theta)=0$,
    then the proof of \eqref{yu-7-16-7} is completed based on \eqref{yu-7-16-3} and \eqref{yu-7-16-2}. If $l(\theta)\geq 1$, it is evident that $[\theta,T]$ can be decomposed as
    $[\theta,T]=[\theta,T-l(\theta)r_0]
    \cup(\cup_{i=1}^{l(\theta)}[T-ir_0,T-(i-1)r_0])$. In this case, we can replace $T$ and $F_{l,N}(x^{s,\zeta}(T))$ with  $T-r_0$ and $y^{s,\zeta}(T-r_0)$ respectively
     (noting that, by \eqref{yu-7-16-3} and \eqref{yu-7-16-2}, $\mathbb{D}_\theta^jy^{s,\zeta}(T-r_0)\in L^2_{\mathcal{F}_{T-r_0}^{w,s}}(\Omega,\mathbb{P};\mathbb{R}^n)$).
     Using the same reasoning, we can demonstrate that $(y^{s,\zeta}, Y^{s,\zeta})\in L^2(T-2r_0,T-r_0;\mathcal{D}_{\mathfrak{T}^s(w,\mathbb{P}),n}^{1,2}(s,T))
          \times (L^2(T-2r_0,T-r_0;\mathcal{D}_{\mathfrak{T}^s(w,\mathbb{P}),n}^{1,2}(s,T)))^d$,
\begin{eqnarray*}\label{yu-7-16-4}
    (\mathbb{D}_\theta^jy^{s,\zeta}, \mathbb{D}_\theta^jY^{s,\zeta})
    \in L^2(\Omega\times((T-2r_0)\vee\theta,T-r_0), \mathbb{P}\times \lambda;\mathbb{R}^{n})\nonumber\\
    \times L^2(\Omega\times((T-2r_0)\vee\theta,T-r_0), \mathbb{P}\times \lambda;\mathbb{R}^{n\times d})
\end{eqnarray*}
    and
\begin{eqnarray*}\label{yu-7-16-5}
     \mathbb{D}_\theta^jy^{s,\zeta}(t)
    &=&\mathbb{D}_\theta^jy^{s,\zeta}(T-r_0)+\int_t^{T-r_0}\Big(A^\top \mathbb{D}_\theta^jy^{s,\zeta}(\sigma)+\sum_{i=1}^dC_i^\top \mathbb{D}_\theta^jY_i^{s,\zeta}(\sigma)\Big)d\sigma\nonumber\\
    &\;&-\sum_{i=1}^d\int_t^{T-r_0}\mathbb{D}_\theta^jY^{s,\zeta}_i(\sigma)dw^i(\sigma)
    \;\;\mbox{for any}\;\;t\in
    \begin{cases}[T-2r_0,T-r_0],&\mbox{if}\;\;l(\theta)>1,\\
    [\theta,T-r_0],&\mbox{if}\;\;l(\theta)=1.
    \end{cases}
\end{eqnarray*}
     Thus, employing an inductive argument and \eqref{yu-7-15-1}, we can establish that \eqref{yu-7-16-7} holds. Leveraging \eqref{yu-7-16-7}, Lemma \ref{yu-lemma-7-23-1}, and the facts that  $\mathbb{D}^{j}_\theta y^{s,\zeta}(\sigma)=0$ and $\mathbb{D}^{j}_\theta Y_i^{s,\zeta}(\sigma)=0$ for any $\sigma\in[s,\theta)$ when $\theta>s$, we can  conclude that  $(a)$ is true.

    {\bf Step 2: The proof of $(b)$.} Utilizing \eqref{yu-7-15-1}, Claim (a), and Lemma \ref{yu-lemma-7-23-1} (applied to \eqref{yu-7-16-7}), we can readily deduce that, for
    every  $p\geq 1$,
\begin{eqnarray}\label{yu-7-18-1}
   \sum_{j=1}^dE_{\mathbb{P}}\int_s^T\Big(\int_s^T\sum_{i=1}^n|\mathbb{D}^j_\theta Y_i^{s,\zeta}(t)|^2_{\mathbb{R}^{n}}dt\Big)^{\frac{p}{2}}d\theta
    =\sum_{j=1}^dE_{\mathbb{P}}\int_s^T
    \Big(\int_{\theta}^T\sum_{i=1}^d|\mathbb{D}^j_\theta Y_i^{s,\zeta}(t)|^2_{\mathbb{R}^{n}}dt\Big)^{\frac{p}{2}}d\theta
       \nonumber\\
       \leq c(p,s,T)\|h^\top\|^p_{C([s,T];\mathbb{R}^{(nl)\times d})}.
\end{eqnarray}
       Let $\psi(t):=\sum_{i=1}^d\int_t^TY^{s,\zeta}_i(\sigma)dw^i(\sigma)$ for $t\in[s,T]$.
     It is straightforward to verify that, for any $(\theta,t)\in[s,T]\times[s,T]$,
\begin{equation*}
     \sum_{j=1}^d|\mathbb{D}_\theta^j\psi(t)|_{\mathbb{R}^{n}}^2
     \leq c\sum_{i=1}^d\Big(|Y_i^{s,\zeta}(\theta)|^2_{\mathbb{R}^n}
     +\sum_{j=1}^d\Big|\int_t^T\mathbb{D}_\theta^jY_i(\sigma)dw^i(\sigma)
     \Big|_{\mathbb{R}^n}^2\Big).
\end{equation*}
         By applying the Bennett-Durrett-Gundy inequality, for each $p\geq 2$, we obtain
\begin{eqnarray*}\label{yu-10-22-b-3}
    &\;&E_{\mathbb{P}}\Big(\sum_{j=1}^d
    \int_s^T|\mathbb{D}^j_\theta\psi(t)|_{\mathbb{R}^n}^2
    d\theta\Big)^\frac{p}{2}\nonumber\\
    &\leq& c(p,s,T)\sum_{i=1}^d E_{\mathbb{P}}
    \Big[\Big(\int_s^T|Y_i^{s,\zeta}(\theta)|^2_{\mathbb{R}^n}
    d\theta\Big)^{\frac{p}{2}}
    +\sum_{j=1}^d\int_s^T\Big|\int_t^T\mathbb{D}^j_\theta Y_i^{s,\zeta}(\sigma)dw^i(\sigma)\Big|^p_{\mathbb{R}^n}d\theta\Big]
    \nonumber\\
    &\leq& c(p,s,T)\sum_{i=1}^d E_{\mathbb{P}}
    \Big[\Big(\int_s^T|Y_i^{s,\zeta}(\theta)|^2_{\mathbb{R}^n}
    d\theta\Big)^{\frac{p}{2}}
    +\sum_{j=1}^d\int_s^T\Big(\int_s^T|\mathbb{D}^j_\theta Y_i^{s,\zeta}(\sigma)|_{\mathbb{R}^n}^2d\sigma\Big)^{\frac{p}{2}} d\theta\Big].
\end{eqnarray*}
    Thus, combining Lemma \ref{yu-lemma-7-23-1} and \eqref{yu-7-18-1}, we conclude that
    $E_{\mathbb{P}}\int_s^T\big(\sum_{j=1}^d\int_s^T
    |\mathbb{D}^j_\theta\psi(t)|^2_{\mathbb{R}^n}d\theta\big)^{\frac{p}{2}}dt<+\infty$.
     Therefore, according to \cite[Proposition 3.2.2, Chapter 3]{Nualart}, we can assert that
    $\sum_{i=1}^d\int_t^TY^{s,\zeta}_i(\sigma)dw^i(\sigma)$  possesses a continuous version. Furthermore, by Lemma \ref{yu-lemma-7-23-1}, it is evident that the term $\int_t^T\big(A^\top y^{s,\zeta}(\sigma)+\sum_{i=1}^dC_i^\top Y_i^{s,\zeta}(\sigma)\big)
    d\sigma$ also has a continuous version. It follows that $y^{s,\zeta}(\cdot)$ has a continuous version over $[s,T]$. By substituting $s$ and $F_{l,N}(x^{s,\zeta}(T))$ with $\theta$ and  $F_{l,N}(x^{s,\zeta}(T))
    \mathbb{D}^j_{\theta}
    x^{s,\zeta}(T)$ respectively, and utilizing \eqref{yu-7-16-7}, we can replicate the aforementioned procedure to demonstrate that $\mathbb{D}^{j}_\theta y^{s,\zeta}(\cdot)$
  possesses a continuous version over $[\theta,T]$ for each $j\in\{1,2,\ldots,d\}$. Therefore, the claim $(b)$ is true.

    {\bf Step 3: The proof of $(c)$.} Indeed, for each $\theta\in(s,T]$,
    based on \eqref{yu-7-16-6} and the equation
\begin{eqnarray*}
    y^{s,\zeta}(t)
    =y^{s,\zeta}(s)
    -\int_s^t\Big(A^\top y^{s,\zeta}(\sigma)+\sum_{i=1}^dC_i^\top Y^{s,\zeta}_i(\sigma)\Big)d\sigma
    +\sum_{i=1}^d\int_s^tY_i^{s,\zeta}(\sigma)dw^i(\sigma)\;\;\mbox{for any}\;\;t\in(\theta,T],
\end{eqnarray*}
    we can deduce that, for each $j\in\{1,2,\ldots,d\}$ and any $\theta\in(s,T]$,
\begin{eqnarray}\label{yu-7-16-20}
     \mathbb{D}^{j}_\theta y^{s,\zeta}(t)&=&Y_j^{s,\zeta}(\theta)
     -\int_{\theta}^t\Big(A^\top \mathbb{D}^{j}_\theta y^{s,\zeta}(\sigma)+\sum_{i=1}^dC_i^\top \mathbb{D}^{j}_\theta Y^{s,\zeta}_i(\sigma)\Big)d\sigma\nonumber\\
     &\;&+\sum_{i=1}^d\int_{\theta}^t\mathbb{D}^{j}_\theta Y_i^{s,\zeta}(\sigma)dw^i(\sigma)\;\;\mbox{for any}\;\;t\in(\theta,T].
\end{eqnarray}
       By Claim (b) and taking the limit as $t\to \theta^+$  in \eqref{yu-7-16-20}, we can conclude that  $\mathbb{D}^{j}_\theta y^{s,\zeta}(\theta)=Y_j^{s,\zeta}(\theta)$ $\mathbb{P}$-almost surely for any $\theta\in[s,T]$.
       The proof of Lemma \ref{yu-lemma-7-23-2} is complete.
\end{proof}

\vskip 5pt

\begin{proof}[\emph{{\bf Proof of  Lemma \ref{yu-proposition-7-11-1}}}]
      The proof is organized into two steps.

    {\bf Step 1: The function $\zeta\to (y^{s,\zeta}, Y^{s,\zeta})$  is second-order continuously differentiable $\mathbb{P}\times \lambda$-almost surely, and satisfies the equation
\begin{equation}\label{yu-7-20-bbb00}
    \nabla y^{s,\zeta}(t)=\nabla F_{l,N}(x^{s,\zeta}(T))
        +\int_t^T\Big(A^\top\nabla y^{s,\zeta}(\sigma)
        +\sum_{i=1}^d C^\top_i\nabla Y_i^{s,\zeta}(\sigma)\Big)d\sigma
        -\sum_{i=1}^d\int_t^T \nabla Y_i^{s,\zeta}(\sigma)dw^{i}(\sigma).
\end{equation}
     Furthermore,
$\nabla y^{s,\zeta}(\cdot)$  has a continuous versions over $[s,T]$.}
      For a function  $g=(g_1,g_2,\cdots,g_n)^\top\in C^1(\mathbb{R}^{nl};\mathbb{R}^{n})$,
       for each  $j\in\{1,2,\ldots, nl\}$, we define the difference operator at
       $\xi\in\mathbb{R}^{nl}$ along the direction $e_j$   with step size
       $ h$ (where $h>0$) as
       $\triangle_h^j g(\xi):=h^{-1}(g(\xi+he_j)-g(\xi))$, where $\{e_j\}_{j=1}^{nl}$ represents the $j$-th vector in the standard orthonormal basis of  $\mathbb{R}^{nl}$.
       It is evident that  $\nabla g(\xi)e_j=\lim_{h\to 0}(\triangle_h^j g(\xi))_{i=1}^d$ for each $j\in\{1,2,\ldots,nl\}$.
 Given the following facts:
\begin{equation*}\label{yu-7-17-11}
    F_{l,N}(x_1)-F_{l,N}(x_2)=\Big(\int_0^1\nabla F_{l,N}(x_2+\rho(x_1-x_2))d\rho\Big) (x_1-x_2)\;\;\mbox{for any}\;\;x_1,x_2\in\mathbb{R}^{nl}
\end{equation*}
    and
$\triangle_h^jx^{s,\zeta}(t)=e_j$ for each $j\in\{1,2,\cdots,nl\}$,
    one can readily verify that, for every   $t\in[s,T]$,
\begin{eqnarray}\label{yu-7-19-2}
    \triangle_h^j y^{s,\zeta}(t)=G_{N,h,j}^{s,\zeta}e_j
        +\int_t^T\Big(A^\top\triangle_h^j y^{s,\zeta}(\sigma)
        +\sum_{i=1}^d C^\top_i\triangle_h^j Y_i^{s,\zeta}(\sigma)\Big)d\sigma
        -\sum_{i=1}^d\int_t^T \triangle_h^j Y_i^{s,\zeta}(\sigma)dw^{i}(\sigma)
\end{eqnarray}
     for each $j\in\{1,2,\cdots,nl\}$, where $G_{N,h,j}^{s,\zeta}:=\int_0^1\nabla F_{l,N}(x^{s,\zeta}(T)+\rho he_j)d\rho$. Now, we  fix $j\in\{1,2,\cdots,nl\}$ arbitrarily.  By the definition of $F_{l,N}$, it is clear that
$|G_{N,h,j}^{s,\zeta}-G_{N,h',j}^{s,\zeta}|_{\mathbb{R}^{n\times nl}}\leq
    c(N)|h-h'|$ $\mathbb{P}$-a.s. for any $h,h'\in\mathbb{R}^+$.
    This, along with \eqref{yu-7-19-2} and Lemma \ref{yu-lemma-7-23-1},  implies  that
\begin{equation}\label{yu-7-19-4}
    E_{\mathbb{P}}\sup_{s\leq t\leq T}|\triangle_{h}^j y^{s,\zeta}(t)-\triangle_{h'}^j y^{s,\zeta}(t)|^2_{\mathbb{R}^n}
    +\sum_{i=1}^dE_{\mathbb{P}}
    \int_s^T|\triangle_{h}^jY_i^{s,\zeta}(t)
    -\triangle_{h'}^jY_i^{s,\zeta}(t)|^2_{\mathbb{R}^n}dt
    \leq c(N)|h-h'|^2
\end{equation}
    for any $h,h'\in\mathbb{R}^+$. It follows that, for any
    sequence  $\{h_k\}_{k\in\mathbb{N}}\subset\mathbb{R}^+$ with $\lim_{k\to+\infty}h_k=0$,
    there exists a tuple  $(z^j(\cdot;\zeta),Z^j(\cdot;\zeta):=(Z^j_i(\cdot;\zeta))_{i=1}^d)\in L^2_{\mathfrak{T}^s(w,\mathbb{P})}(\Omega;C([s,T];\mathbb{R}^n))
     \times (L^2_{\mathfrak{T}^s(w,\mathbb{P})}(s,T;\mathbb{R}^n))^d$ such  that
\begin{equation}\label{yu-7-19-5}
    \lim_{k\to+\infty}
    \Big(E_{\mathbb{P}}\sup_{s\leq t\leq T}|\triangle_{h_k}^j y^{s,\zeta}(t)-z^j(t;\zeta)|^2_{\mathbb{R}^n}
    +\sum_{i=1}^dE_{\mathbb{P}}
    \int_s^T|\triangle_{h_k}^jY_i^{s,\zeta}(t)
    -Z^j_i(t;\zeta)|^2_{\mathbb{R}^n}dt\Big)=0.
\end{equation}
    Moreover, by the dominated convergence theorem, we have
$\lim_{h\to0^+}E_{\mathbb{P}}|G_{N,h,j}^{s,\zeta}-\nabla F_{l,N}(x^{s,\zeta}(T))|^2_{\mathbb{R}^{n\times nl}}=0$.
    It follows from \eqref{yu-7-19-2} and \eqref{yu-7-19-5} that
    the tuple $(z^j(\cdot;\zeta),Z^j(\cdot;\zeta))$ satisfies
\begin{equation*}\label{yu-7-20-1}
    z^j(t;\zeta)=\nabla F_{l,N}(x^{s,\zeta}(T))e_j
        +\int_t^T\Big(A^\top z^j(\sigma,\zeta)
        +\sum_{i=1}^d C^\top_iZ^j_i(\sigma,\zeta)\Big)d\sigma
        -\sum_{i=1}^d\int_t^T Z_i^j(\sigma,\zeta)dw^{i}(\sigma).
\end{equation*}
     Furthermore, by \eqref{yu-7-19-4} and \eqref{yu-7-19-5}, it is evident that
\begin{equation*}\label{yu-7-20-2}
    \lim_{h\to0}
    \Big(E_{\mathbb{P}}\sup_{s\leq t\leq T}|\triangle_{h}^j y^{s,\zeta}(t)-z^j(t,\zeta)|^2_{\mathbb{R}^n}
    +\sum_{i=1}^dE_{\mathbb{P}}
    \int_s^T|\triangle_{h}^jY_i^{s,\zeta}(t)
    -Z^j_i(t,\zeta)|^2_{\mathbb{R}^n}dt\Big)=0.
\end{equation*}
     This, combined with the arbitrariness of  $j$, implies  that $\nabla y^{s,\zeta}(\cdot)=z(\cdot;\zeta):=(z^j(\cdot;\zeta))_{j=1}^n$ and
    $\nabla Y_i^{s,\zeta}(\cdot)=Z_i(\cdot;\zeta):=(Z_i^j(\cdot;\zeta))_{j=1}^n$ hold  $\mathbb{P}\times \lambda$-almost surely for each  $i\in\{1,2,\ldots,d\}$.
    Therefore, the function $\zeta\to (y^{s,\zeta}, Y^{s,\zeta})$ is first-order differentiable $\mathbb{P}\times \lambda$-almost surely, and \eqref{yu-7-20-bbb00} holds. Applying the same reasoning to  $(z^j(\cdot;\zeta),Z^j(\cdot;\zeta))$, we can demonstrate  that $(z^j(\cdot;\zeta),Z^j(\cdot;\zeta))$ is also first-order differentiable $\mathbb{P}\times \lambda$-almost surely. Consequently,   $\zeta\to (y^{s,\zeta}, Y^{s,\zeta})$ is second-order differentiable $\mathbb{P}\times \lambda$-almost surely.  The continuity of
    $\zeta\to (y^{s,\zeta}, Y^{s,\zeta})$ and its first and second-order differentiations are straightforward to prove due to the smoothness of
     $F_{l,N}$.

    By replacing  $F_{l,N}(x^{s,\zeta}(T))$ with   $\nabla F_{l,N}(x^{s,\zeta}(T))e_j$,  and following the same reasoning used in the proof of part $(b)$ of Lemma \ref{yu-lemma-7-23-2}, one can easily verify that  $\nabla y^{s,\zeta}(\cdot)e_j$ has a continuous version over $[s,T]$.

    {\bf Step 2: Completion of  the proof.}
    From part (a) of Lemma \ref{yu-lemma-7-23-2}, the claim in {\bf Step 1}, and the uniqueness of the solution to the FBSDE \eqref{yu-7-9-3}, we can conclude that
$\mathbb{D}_\theta y^{s,\zeta}(t)=\nabla y^{s,\zeta}(t)h^{\top}(\theta)$
    for any $0\leq \theta\leq t\leq T$.
     This, combined with part $(c)$ of Lemma \ref{yu-lemma-7-23-2}, implies that
$Y^{s,\zeta}(t)=\nabla y^{s,\zeta}(t)h^{\top}(t)$ for a.e. $t\in[s,T]$.

{\bf Substep 2.1: Show  that $U(s,\zeta):=y^{s,\zeta}(s)$ is in $C^{1,2}([0,T]\times \mathbb{R}^{nl};\mathbb{R}^n)$, bounded and satisfies the equation \eqref{yu-7-22-4}.} By Lemma \ref{yu-lemma-7-23-1}, the claim in {\bf Step 1}, and the definition of
     $F_{l,N}$, it can be readily shown that $U(\cdot,\cdot)$ is bounded and in $C^{0,2}([0,T]\times \mathbb{R}^{nl};\mathbb{R}^n)$.  Other conclusions have been established in \cite[Theorem 3.2]{Pardoux-Peng-1992} using the  It\^{o} formula and the claims in {\bf Step 1}. Therefore, we omit the proof here.

    {\bf Substep 2.2:  Demonstrate the uniqueness of the solution to equation \eqref{yu-7-22-4} and establish that if $U(\cdot,\cdot)\in C^{1,2}([0,T]\times\mathbb{R}^{nl};\mathbb{R}^n)$ is the solution of the equation \eqref{yu-7-22-4}, then  $(y^{s,\zeta}(t),Y^{s,\zeta}(t)):=(U(t,x^{s,\zeta}(t)),\nabla U(t,x^{s,\zeta}(t))h^\top(t))$ solves the BSDE in \eqref{yu-7-9-3}.}
     We first demonstrate that if  $U(\cdot,\cdot)\in C^{1,2}([0,T]\times\mathbb{R}^{nl};\mathbb{R}^n)$ is the solution to the equation \eqref{yu-7-22-4}, then
\begin{equation}\label{yu-7-23-0001}
    (y(\cdot),Y(\cdot)):=(U(\cdot,x^{s,\zeta}(\cdot)),\nabla U(\cdot,x^{s,\zeta}(\cdot))h^\top(\cdot))
\end{equation}
       is the solution to the BSDE in \eqref{yu-7-9-3}.

From \eqref{yu-7-22-4}, it is evident that
\begin{eqnarray}\label{yu-7-23-10}
   U(T,x^{s,\zeta}(T))=F_{l,N}(x^{s,\zeta}(T)).
\end{eqnarray}
    Applying the It\^{o} formula to $U(t,x^{s,\zeta}(t))$, and using  \eqref{yu-7-22-4}, we obtain
\begin{eqnarray*}\label{yu-7-23-11}
 dU(t,x^{s,\zeta}(t))=\Big(-A^\top U(t,x^{s,\zeta}(t))-\sum_{i=1}^d
 C_i^\top \nabla U(t,x^{s,\zeta}(t))h^\top (t)f_i\Big)dt
 +\nabla U(t,x^{s,\zeta}(t))h^\top (t) dw(t).
\end{eqnarray*}
     Combining this with \eqref{yu-7-23-10} and \eqref{yu-7-23-0001}, we conclude tha $(x^{s,\zeta}(\cdot),y(\cdot),Y(t))$  is the solution to
    the equation \eqref{yu-7-9-3}.

    Next, we prove uniqueness. Suppose that \eqref{yu-7-22-4} has two distinct  solutions $U$ and $\widehat{U}$ in $C^{1,2}([0,T]\times\mathbb{R}^{nl};\mathbb{R}^n)$.
    Then, there exists  $(s_0,\zeta_0)
    \in [0,T]\times\mathbb{R}^n$ such  that $U(s_0,\zeta_0)\neq \widehat{U}(s_0,\zeta_0)$.   By the previous claim, this implies that both
    $(y^{s_0,\zeta_0}(t),Y^{s_0,\zeta_0}(t)):=(U(t,x^{s_0,\zeta_0}(t)),\nabla U(t,x^{s_0,\zeta_0}(t))h^\top(t))$ and  $(\widehat{y}^{s_0,\zeta_0}(t),\widehat{Y}^{s_0,\zeta_0}(t)):=
     (\widehat{U}(t,x^{s_0,\zeta_0}(t)),\nabla \widehat{U}(t,x^{s_0,\zeta_0}(t))h^\top(t))$
     are  the solutions   to  \eqref{yu-7-9-3}
     with   initial value $\zeta_0$ at  initial time $s_0$. However, since $U^{s_0,\zeta_0}(s_0)\neq \widehat{U}^{s_0,\zeta_0}(s_0)$, it follows that $y^{s_0,\zeta_0}(s_0)\neq \widehat{y}^{s_0,\zeta_0}(s_0)$.
      This contradicts the uniqueness of the solution to \eqref{yu-7-9-3}, thereby proving that the solution to \eqref{yu-7-22-4} is unique.
\end{proof}

\vskip 5pt
    \textbf{Acknowledgments.} The authors would like to express their  appreciation to Dr. Tianxiao Wang of Sichuan University for his valuable comments and suggestions during the preparation of this manuscript.

\end{document}